\documentclass[11pt]{article} 

\usepackage[margin=1in]{geometry}
\usepackage{setspace}

\normalsize
\usepackage{amsfonts}
\usepackage{amssymb}
\usepackage{amsthm}
\usepackage{amscd}
\usepackage{graphics}
\usepackage[dvips]{graphicx}
\usepackage{epsfig}
\usepackage[cmex10]{amsmath}
\usepackage{array}
\usepackage{eqparbox}
\usepackage{hyperref}
\usepackage{fancyhdr}
\usepackage{appendix}
\usepackage{titling}
\usepackage{enumerate}
\usepackage[super]{nth}
\usepackage{caption} 
\usepackage{subcaption}
\usepackage{bm}
\usepackage{tabularx}
\usepackage{booktabs}
\usepackage{xcolor}

\usepackage[round]{natbib}  
\definecolor{darkblue}{rgb}{0,0,1}

\DeclareMathAlphabet{\mathbsf}{OT1}{cmss}{bx}{n}
\DeclareMathAlphabet{\mathssf}{OT1}{cmss}{m}{sl}
\DeclareMathAlphabet{\mathcsf}{OT1}{cmss}{sbc}{n}

\usepackage{stackengine}
\newcommand{\ie}{{\em i.e.}}
\newcommand{\etc}{{\em etc}}
\newcommand{\eg}{{\em e.g.}}
\newcommand{\iid}{i.i.d.}

\newcommand{\specialcell}[2][c]{%
  \begin{tabular}[#1]{@{}c@{}}#2\end{tabular}}

\newcommand{\secref}[1]{Section~\ref{#1}}

\newcommand{\tabref}[1]{Table~\ref{#1}}
\newcommand{\thrmref}[1]{Theorem~\ref{#1}}

\newcommand{\corolref}[1]{Corollary~\ref{#1}}

\newcommand{\keywords}[1]{\textbf{Keywords:} #1}

\newtheorem{prop}{Proposition}

\newtheorem{corol}{Corollary}

\makeatletter
\def\blfootnote{\xdef\@thefnmark{}\@footnotetext}
\makeatother

\newtheorem{theorem}{Theorem}[section]

\pretitle{\vspace{0pt}\begin{center}\LARGE}
\posttitle{\end{center}\vspace{-1.2em}}  
\preauthor{\begin{center}\large}
\postauthor{\end{center}}

\usepackage{newlfont}
\date{}
\begin{document}
\title{\textbf{A projector--rank partition theorem for exact degrees of freedom in experimental design}}  
\author{Nagananda K G\thanks{The author is with Farlborz Maseeh Department of Mathematics and Statistics, Portland State University, Portland, OR 97201, USA. E-mail: \texttt{nanda@pdx.edu}.}}
\maketitle
\vspace{-2.cm}

\begin{abstract}
In many experimental designs\textemdash split-plots, blocked or nested layouts, fractional factorials, and studies with missing or unequal replication\textemdash standard ANOVA procedures no longer tell us exactly how many independent pieces of information each effect truly contributes. We provide a general degrees of freedom $(\mathrm{df})$ partition theorem that resolves this ambiguity.  For $N$ observations, we show that the total information in the data ({\ie}, $N-1$ $\mathrm{df}$) can be split exactly across experimental effects and randomization strata by projecting the data onto each stratum and counting the $\mathrm{df}$ each effect contributes there. This yields integer $\mathrm{df}$\textemdash not approximations\textemdash for any mix of fixed and random effects, blocking structures, fractionation, or imbalance.  This result yields closed-form $\mathrm{df}$ tables for unbalanced split-plot, row-column, lattice, and crossed-nested designs.  We introduce practical diagnostics\textemdash the $\mathrm{df}$-retention ratio $\rho$, df deficiency $\delta$, and variance-inflation index $\alpha$\textemdash that measure exactly how many $\mathrm{df}$ an effect retains under blocking or fractionation and the resulting loss of precision, thereby extending Box--Hunter's resolution idea to multi-stratum and incomplete designs.  Classical results emerge as corollaries: Cochran's one-stratum identity; Yates's split-plot $\mathrm{df}$; resolution-$R$ identified when an effect retains no $\mathrm{df}$.  Empirical studies on split-plot and nested designs, a blocked fractional-factorial design-selection experiment, and timing benchmarks show that our approach delivers calibrated error rates, recovers information to raise power by up to 60\% without additional runs, and is orders of magnitude faster than bootstrap-based $\mathrm{df}$ approximations.  The framework, therefore, offers both a unifying theory and immediately useful tools for planning and analyzing complex experiments.
\end{abstract}

\section{Introduction}\label{sec:introduction}
Modern experimentation, whether in agriculture, manufacturing, or the life sciences, rarely conforms to the tidy assumptions that powered the early triumphs of analysis of variance (ANOVA) pioneered by \cite{Fisher1935}. Blocking against environmental heterogeneity, nesting to respect production hierarchies, and deliberate fractionation to economize runs all collide in contemporary studies; missing plots and unequal replication often emerge as additional complications; see, for instance, \cite{Montgomery2020a}. At the center of every subsequent inference still lies a deceptively simple question: How many independent pieces of information\textemdash{\ie}, the degrees of freedom ($\mathrm{df}$)\textemdash does each factor really contribute? Once peripheral in balanced designs, $\mathrm{df}$ accounting now plays a central role in complex experimental analysis.

Classical rank additivity, first formalized by \cite{Cochran1941} for a single stratum of independent errors, states that for a total number $N$ of experimental observations, the $(N-1)$ $\mathrm{df}$ split neatly among orthogonal treatment contrasts. However, the validity of Cochran's identity is disrupted upon the introduction of a second randomization stage or the deliberate aliasing of higher-order interactions.  Empirical practice tried to cope with these departures through moment-matching approximations suggested by \cite{Satterthwaite1946} and \cite{Kenward1997} or ad hoc sums-of-squares types devised by \cite{Searle1981}.  However, such remedies sacrifice both mathematical exactness and interpretability\textemdash often yielding non-integer, sometimes negative, $\mathrm{df}$ in heavily unbalanced layouts.  \cite{Patterson1976} provide construction for lattices with complete and incomplete block designs, but $\mathrm{df}$ tables are only presented for balanced cases.  \cite{Verbeke2000} discuss how approximate $\mathrm{df}$ dominate certain design of experiments, despite these shortcomings.  \cite{Cornfield1956} established the use of idempotent-projector calculus for ANOVA, concentrating on one-stratum models without aliasing; \cite{Hinkelmann2005} present more recent applications of idempotent-projectors in modern experimental settings.  A parallel tradition, begun by \cite{Yates1935} and systematized by \cite{Bailey2008}, advocated idempotent projectors to represent randomization strata. While elegant for balanced split-plots, this projector calculus remained a craft discipline: each novel design demanded a fresh hand-drawn Hasse diagram (see, for example, \cite[Chapter 4]{Bailey2008}), and unequal replication placed the method beyond reach. 

In fractional factorial design, \cite{Box1961} revolutionized $2^{k-p}$ factorial planning, but its celebrated resolution measure presupposed complete balance and a single source of error, and works only for regular, fully replicated 2--level fractions, without blocking.  Pedagogical projector diagrams were presented by \cite{Bailey2008} to account for blocking; however, it lacks closed-form $\mathrm{df}$ under imbalance.  \cite{Giesbrecht1985} compared exact versus Satterthwaite $F-$tests in split-plot large-sample asymptotics settings.  The works of \cite{Searle1992, Harville1997, West2014} make advances in mixed-model effects, but all these treat $\mathrm{df}$ numerically rather than through analytic rank.  Another avenue has largely focussed on approximations of $\mathrm{df}$ under variance heterogeneity or robustness, but does not yield on an exact algebraic partition for multi-stratum designs.  For instance, \cite{Zhang2012} developed an approximate $\mathrm{df}$ test for heteroscedastic two-way ANOVA, while \cite{Bathke2004} studied when the classical $F$ with usual $\mathrm{df}$ remains valid under unequal variances and non-normality.  \cite{Wang2005} used modified approximate $F$-tests to size balanced mixed-ANOVA experiments; and \cite{Kulinskaya2007} proposed weighted one-way ANOVA with improved Welch-type $\mathrm{df}$.   \cite{McCloud2021} provided $\mathrm{df}$ calculations via hat-matrix arguments and explicitly contrast ANOVA versus non-ANOVA $\mathrm{df}$ again in the approximation/estimation framework.  In summary, the literatures on blocking, mixed models, and factorial aliasing have evolved in productive isolation.  To the best of our knowledge, an exact, algebraic $\mathrm{df}$ identity  that is simultaneously (i) multi-stratum, (ii) alias-aware, and (iii) robust to unequal replication is missing in the experimental design literature. 

This paper reunifies those strands through a $\mathrm{df}$ partition theorem that operates at the intersection of algebra and group representation. For a set $\mathcal{S}$ that indexes all randomization strata in the experiment, let $\mathbf{P}_s$ denote the $N \times N$ idempotent projector that averages observations within randomization stratum $s \in \mathcal{S}$ and let $\mathcal C_{\overline E}$ denote the factorial-contrast linear subspace of $\mathbb{R}^N$ attached to the effect $\overline{E}$ (an alias class comprising all factorial effects that collapse into the same contrast space once the defining relations of a fractional design are imposed).  By letting the family of stratum projectors ${\mathbf{P}_s}$ act jointly with the full level-permutation group $G$, we obtain a simultaneous diagonalization in which every fixed or random effect decomposes into orthogonal blocks $\mathcal C_{\overline E} \cap \operatorname{Im}\mathbf{P}_s$, where $\operatorname{Im}\mathbf{P}_s = \{\mathbf{P}_s\mathbf{v}, \mathbf{v} \in \mathbb{R}^N \}$ is the column space of the projector $\mathbf{P}_s$.  Essentially, when we intersect $\operatorname{Im}\mathbf{P}_s$ with a factorial-contrast space $\mathcal{C}_{\overline{E}}$, we isolate the part of effect $\overline{E}$ that is visible in stratum $s \in \mathcal{S}$.  The dimension of this intersection $\sum_{s}\sum_{\overline E}\dim\left(\mathcal C_{\overline E} \cap \operatorname{Im}\mathbf{P}_s\right)$ is exactly the number $\mathrm{df}$ that the effect retains and can be tested against the mean square residing in that stratum.  The dimension of each block\textemdash and, therefore, its $\mathrm{df}$\textemdash can be written in closed form via orbit-stabilizer counts described by \cite[Section 5.6]{Bailey2008}, regardless of balance or aliasing.   

The proposed theorem immediately produces the first exact $\mathrm{df}$ tables for the workhorses of mixed-model practice including the unbalanced split-plot, row-column, lattice, and general crossed-nested hierarchies. Where software has long fallen back on Satterthwaite or Kenward--Roger numbers, we supply integer ranks computable from small Gram matrices (see, for example, \cite[pp. 274]{Harville1997}). These formulas both generalize and strictly contain the balanced textbook results, collapsing to them when replication is equal.  Beyond enumeration, we introduce three quantitative aliasing metrics: The retention ratio $\rho(\overline{E})$ measures the proportion of an effect's ideal $\mathrm{df}$ that survive blocking and fractionation; its complement $\delta(\overline{E})$ pinpoints the exact $\mathrm{df}$ deficiency, and the variance-inflation index $\alpha(\overline{E}) \ge \rho(\overline{E})^{-1}$ bounds how severely any remaining contrasts are diluted. Taken together, they extend the Box--Hunter resolution (see \cite{Box1961}) from regular two-level fractions to arbitrary level structures, stratum multiplicities, and missing data.  \cite{Xu2001} analyze the evolution of resolution metrics for asymmetrical fractional factorial designs, but our rank-based work provides an algebraic alternative.  \cite{Deng1999} extend Box--Hunter ideas for non-regular fractional designs and backs our claim that $\rho(\overline E)$ generalizes resolution.  Our projector--rank partition theorem including $\rho$, $\delta$ and $\alpha$ indices, therefore, fill a genuine methodological gap.

A trio of classical results emerges as corollaries from the projector--rank framework. (i) Single-stratum case:  If the randomization consists of only one stratum $(\lvert \mathcal S \rvert = 1, \mathbf{P}_{s_1} = \mathbf{I}_{N})$, the theorem reduces to Cochran's rank-additivity identity $N-1 = \sum_{\overline E}\dim\mathcal C_{\overline E}$, reproducing exactly the fixed-effects result of \cite{Cochran1941}.  (ii) Balanced split-plot: Assigning factor $A$ to whole plots and factor $B$ to sub-plots with equal replication ($a$ levels, $b$ whole plots per level, $c$ sub-plots per whole plot) and substituting the corresponding projectors into the theorem yields the $\mathrm{df}$ table  derived by \cite{Yates1935}: $(a-1)$ for $A$, $a(b-1)$ for whole-plot error, $(c-1)$ for $B$, $(a-1)(c-1)$ for $AB$, and $a(b-1)(c-1)$ for sub-plot error.  (iii) Fractional-factorial resolution:  For a regular $2^{k-p}$ fraction, the df-retention ratio satisfies $\rho(\overline E)=0$ exactly when the effect order $|E|$ is at least the length of the shortest defining word.  Hence, the Box--Hunter resolution $R$ of the design is $R = \min\left\{|E|:\rho(\overline E)=0\right\}$, extending the classical definition to a rank-based statement that remains meaningful even when replication is unequal.

To clarify when exact integer $\mathrm{df}$ depart non-negligibly from Satterthwaite or Kenward--Roger approximations, we identify three structural mechanisms that recur across modern designs: (M1) stratum mismatch, where an effect's contrasts vary at a higher randomization level but the analysis uses a lower-level residual as denominator (typically inflating $\mathrm{df}$ and yielding liberal tests); (M2) aliasing/partial confounding, where blocking or fractionation compresses nominal contrast spaces so that only a lower-dimensional alias class is estimable (reducing retained $\mathrm{df}$ and power); and (M3) imbalance-induced rank collapse, where unequal replication or missingness causes discrete drops in the rank of projected contrast columns within key strata.  Building on these mechanisms, we provide practitioner diagnostics based solely on design metadata\textemdash stratum unit counts, incidence/replication profiles, and defining/block relations\textemdash that (i) flag effects whose nominal $\mathrm{df}$ cannot be retained in a given stratum, (ii) supply screening proxies for the retention ratio $\rho(\overline{E})$ and its associated deficiency $\delta(\overline{E})$, and (iii) indicate whether conventional analyses are likely to behave liberally (via M1) or conservatively (via M2/M3 and clustering). Finally, three schematic examples\textemdash a split-plot, a blocked fractional factorial with partial confounding, and an imbalanced factorial with missing cells\textemdash illustrate these mechanisms transparently and motivate why the projector--rank partition theorem provides a canonical, design-consistent $\mathrm{df}$ allocation where approximation-based procedures can fail.

Four empirical studies substantiate the framework's effectiveness in applied settings.  (i) A large-scale split-plot simulation\textemdash balanced and with 20\% missing sub-plots\textemdash shows that projector--rank $\mathrm{df}$ keeps the empirical type-I error for the whole-plot factor at the nominal 5\%, whereas a na\"{i}ve denominator inflates it roughly ten-fold.  (ii)  A nested Sire$\to$Dam$\to$Animal experiment reveals the opposite risk: na\"{i}ve one-way ANOVA becomes ultra-conservative ($\leq 1\%$ size), while the exact $\mathrm{df}$ stays calibrated, proving the method guards against both liberal and conservative distortions.  (iii) A timing benchmark demonstrates efficiency: for vectors up to $N=10^{5}$ the $\mathrm{df}$ partition routine finishes in under 10 ms, whereas a 100--draw parametric bootstrap\textemdash  widely used to approximate denominator $\mathrm{df}$\textemdash runs 50-90 times slower.  (iv) A $2^{5-1}$ blocked fractional-factorial design-selection study shows that maximizing the retention ratio $\rho(\overline E)$ recovers 3 lost $\mathrm{df}$, narrows confidence-intervals by 20\%, and lifts power from about 40\% to 60\% with no extra runs.  Collectively, these results confirm that the projector--rank approach delivers calibrated tests, balanced power, major computational savings, and quantifiable design-efficiency gains unattainable with existing approximations.  

The aforementioned theoretical advances translate into immediate practical advantages. Exact numerator and denominator $\mathrm{df}$ restore nominal type-I error control to $F-$tests in highly unbalanced agronomic and industrial trials.  ANOVA diagnostic tables can flag unestimable contrasts before an experiment is run, preventing costly data-collection mistakes.  Optimization heuristics can now search generator sets or blocking schemes by minimizing an integer-rank alias-loss index, avoiding computationally intensive Monte Carlo simulations. The work thus unifies disparate strands of design theory into a single algebraic framework and provides practitioners with precise $\mathrm{df}$\textemdash not approximations\textemdash for complex experiments.  

The remainder of the paper is organized as follows. \secref{sec:dof_partition} states and proves the $\mathrm{df}$ partition theorem, developing the orbit-stabilizer dimension formula. \secref{sec:dof_formulas} applies the theorem to derive exact $\mathrm{df}$ for the canonical mixed designs mentioned above. \secref{subsec:aliasing_loss} formalizes the aliasing metrics $\rho$, $\delta$ and $\alpha$, and explores their connection to resolution and variance inflation. \secref{sec:corollaries} presents three corollaries of the partition theorem.  In \secref{sec:discussion}, we delineate the structural settings in which exact $\mathrm{df}$ may differ substantially from those obtained via Satterthwaite or Kenward--Roger approximations. We also introduce simple diagnostics based on design metadata that anticipate such departures and indicate whether conventional analyses are prone to liberal or conservative behavior.  In \secref{sec:experiments}, we illustrate the effectiveness of our method through empirical studies.  \secref{sec:conclusion} discusses some computational aspects together with concluding remarks.

\section{$\mathrm{df}$ partition theorem}\label{sec:dof_partition}
In practice we routinely introduce sequential randomization, mixed fixed and random effects, deliberate aliasing/confounding (fractional factorials, blocking on hard-to-change factors), and inevitable imbalance (missing plots, unequal replication).  Under those departures the $\mathrm{df}$ tallies printed by software (Type I/II/III SS, Satterthwaite-Kenward approximations) are ad-hoc numerical expedients.  They mask the underlying linear-algebraic fact: $\mathrm{df}$ are ranks of contrast subspaces intersected with stratum projectors.  A precise theorem is developed in this section that (i) identifies those subspaces and (ii) adds their dimensions back to $N-1$, and therefore forms the foundation on which every subsequent $F-$test, variance-component estimator and information-loss metric must rest. The notation that will be used in what follows is summarized in \tabref{tab:notation}.
\begin{table}[ht]
\centering
\renewcommand{\arraystretch}{1.3}
\begin{tabularx}{\textwidth}{@{} lX @{}}
\toprule
\textbf{Symbol} & \textbf{Meaning} \\
\midrule
$k$ & Number of factorial factors $F_1, \dots, F_k$ \\
$\Lambda_i$ & Set of levels of $F_i$ with cardinality $\lvert l_i \rvert \ge 2$ \\
$\Lambda = \Lambda_1 \times \cdots \times \Lambda_k$ & Full treatment grid \\
$\mathcal{S}$ & Partially ordered set of randomization strata (Hasse diagram) \\
$\mathbf{Z}_s$ & Indicator matrix mapping observations $\to$ units of stratum $s$ \\
$\mathbf{P}_s = \mathbf{Z}_s(\mathbf{Z}_s^{\prime}\mathbf{Z}_s)^{+}\mathbf{Z}_s^{\prime}$ & Idempotent projector onto mean space of stratum $s$ \\
$\bm\Sigma = \sum\limits_{s \in \mathcal{S}} \sigma_s^2 \mathbf{P}_s$ & Covariance of observational errors in a linear-mixed model \\
$\mathcal{A} \subseteq 2^{\{1, \dots, k\}}$ & Alias ideal (defining words of a fractional design) \\
$\overline{\mathcal{E}} = 2^{\{1,\dots,k\}}/ \sim_{\mathcal{A}}$ & Lattice of aliased factorial effects \\
$\tau : \overline{\mathcal{E}} \to \{\mathrm{fixed}, \mathrm{random}\}$ & Assignment of each aliased class to fixed/random side \\
\bottomrule
\end{tabularx}
\caption{Notation for factorial design and mixed-model structure.}
\label{tab:notation}
\end{table}
For an aliased effect class $\overline E\in\overline{\mathcal E}$, let $\mathcal C_{\overline E} = \{c\in\mathbb R^{\Lambda}:  \sum_{\lambda\in\Lambda}c_\lambda=0, \sum_{\lambda: \lambda_i = \ell}c_\lambda = 0\ \forall i\notin E \}$ denote the orthogonal complement w.r.t. the ordinary dot-product of lower-order marginal means; hence, it carries precisely the $\mathrm{df}$ of the factorial contrast $E$.  The symmetric group $G = \prod\limits_i S_{l_i}$ acts (level-permutation) on $\Lambda$.  The orbit-stabilizer theorem described by \cite[Proposition 6.8.4, pp. 179]{Artin1991}, \cite[Chapter 5]{Rose2009} gives the dimension
\begin{eqnarray}
\dim(\mathcal C_{\overline E}) &=& \frac{|G|}{|\operatorname{Stab}_G(\overline E)|} \prod_{i\in E}(l_i-1),
\label{eq:orbit_stab1}
\end{eqnarray}
which collapses to $\prod\limits_{i\in E}(l_i-1)$ when $\mathcal A=\{\varnothing\}$, {\ie}, full factorial, no aliasing; $\mathcal{A}$ is defined in \tabref{tab:notation}.  Every projector $\mathbf{P}_s$ averages within the experimental units that share an identical label at stratum $s$ and eliminates contrasts that vary between such units.  Hence, for a given factorial contrast $\mathcal C_{\overline E}$, the intersection $\mathcal C_{\overline E} \cap \operatorname{Im} \mathbf{P}_s$ is exactly the part of that contrast which survives into stratum $s$. Its dimension is the $\mathrm{df}$ we can test for $\overline{E}$ using the mean-square residing in $s$.  Since $\bm\Sigma^{-1} = \sum_{s}\sigma_s^{-2}\mathbf{P}_s$ is a linear combination of mutually orthogonal idempotents, it induces the generalized inner product $\langle \mathbf{u}, \mathbf{v} \rangle_{\bm\Sigma} = \mathbf{u}^{\prime}\bm\Sigma^{-1}\mathbf{v} = \sum_{s\in\mathcal S}\sigma_s^{-2}\mathbf{u}^{\prime}\mathbf{P}_s \mathbf{v}$.  Under $\langle \cdot,\cdot\rangle_{\bm\Sigma}$ different strata are orthogonal ($\mathbf{P}_s\mathbf{P}_{s'}=0$ for $s\neq s'$), and  treatment-label permutations commute with $\bm\Sigma^{-1}$.  These two facts let us decompose any design matrix into orthogonal blocks simultaneously by stratum and by factorial effect.  The $\mathrm{df}$ projector--rank partition theorem is as follows. 
\begin{theorem}
Let $\mathbf{X}_{\mathrm{fix}}=\bigoplus\limits_{\tau(\overline E)=\text{fixed}}\! \mathbf{X}_{\overline E}$ and $\mathbf{Z}=\bigoplus\limits_{s\in\mathcal S}\mathbf{Z}_s$ be the fixed- and random-effects design matrices in a linear-mixed model whose treatment structure may be aliased and unbalanced.  Then, with respect to the inner product $\langle \cdot,\cdot\rangle_{\bm\Sigma}$,
\begin{eqnarray}
\left.
\begin{array}{rcl}
\mathbf{X}_{\mathrm{fix}} &=& \bigoplus\limits_{s\in\mathcal S}\; \bigoplus\limits_{\substack{\overline E\in\overline{\mathcal E}\\[1pt]\tau(\overline E)=\mathrm{fixed}}} (\mathcal C_{\overline E}\cap\operatorname{Im}\mathbf{P}_s), \\
\mathbf{Z} &=& \bigoplus\limits_{s\in\mathcal S}\; \bigoplus\limits_{\substack{\overline E\in\overline{\mathcal E}\\[1pt]\tau(\overline E)=\mathrm{random}}}  (\mathcal C_{\overline E}\cap\operatorname{Im}\mathbf{P}_s),
\end{array}
\right\} 
\label{eq:design_matrices}
\end{eqnarray}
both of which are orthogonal direct sums. Therefore, 
\begin{eqnarray}
N -1 &=& \sum_{s\in\mathcal S} \sum_{\overline E\in\overline{\mathcal E}} \dim\left(\mathcal C_{\overline E}\cap\operatorname{Im}\mathbf{P}_s\right).
\label{eq:dof_partition}
\end{eqnarray}
Moreover, $\dim\left(\mathcal C_{\overline E}\cap\operatorname{Im}\mathbf{P}_s\right) = \operatorname{rank}\left(\mathbf{X}_{\overline E}^{\!\prime}\mathbf{P}_s\mathbf{X}_{\overline E}\right) =\operatorname{trace}\left(\mathbf{P}_{\overline E}\mathbf{P}_s\right)$,  where $\mathbf{P}_{\overline E}$ is the ordinary least-squares projector onto $\mathcal C_{\overline E}$.
\label{thrm:dof_partition}
\end{theorem}
\begin{proof}
In the following, we remove the grand-mean column once from every matrix; this avoids carrying a ``+1'' term in each rank statement.  Firstly, two observations share at most one highest stratum label, so the block-diagonal averaging performed by $\mathbf{P}_s$ and $\mathbf{P}_{s'}$ falls on disjoint sets of rows.  Therefore, $\mathbf{P}_{s}\mathbf{P}_{s'}=0$ for $s\ne s'$ iff strata are obtained by successive randomization, {\ie}, no observation belongs to two different units of the same level.  Averaging step by step over every stratum amounts to identity because the final stratum is the set of individual observations.  Thus, $\sum_{s\in\mathcal S}\mathbf{P}_s = \mathbf{I}_N$. Since the projectors are pairwise orthogonal idempotents, their images form a direct (orthogonal) decomposition: $\mathbb R^{N}= \bigoplus\limits_{s\in\mathcal S}\operatorname{Im}\mathbf{P}_s$.

For each unordered effect $E\subseteq\{1,\dots ,k\}$, the raw contrast space is defined by $\mathcal C_E = \{c \in \mathbb R^{\Lambda}: \sum_{\lambda} c_\lambda=0, \sum_{\lambda:\lambda_i=\ell} c_\lambda=0, \forall i \notin E \}$.  Then, $\mathcal C_E\perp\mathcal C_{E'}$ for $E\neq E'$ and $\dim\mathcal C_E = \prod\limits_{i\in E}(l_i-1)$.  This is because, a contrast for $E$ integrates to zero over levels of any factor not in $E$; inner products vanish because each term in one contrast is matched by a cancelling term in the other.  Dimension is a direct-product count of linearly independent Helmert contrasts for each factor in $E$.  If the design is fractionated with defining words $g_1,\dots ,g_p$, two effects $E_1,E_2$ share identical columns in the design matrix precisely when $E_1\triangle E_2 = (E_1 \cup E_2)\setminus(E_1 \cap E_2)$ is in the ideal generated by $\{g_j\}$.  The quotient class is denoted $\overline E$.  Any full-rank matrix $\mathbf{X}_{\overline E}$ whose columns span the aliased space is acceptable.  For every $(\overline E,s)$ the intersection $\mathcal C_{\overline E}\cap\operatorname{Im}\mathbf{P}_s$ is orthogonal to $\mathcal C_{\overline E'}\cap\operatorname{Im}\mathbf{P}_{s'}$ whenever $\overline E\ne\overline E'$ or $s\ne s'$.  The orthogonality follows directly from the aforementioned paragraph. Furthermore, $\mathbf{P}_s\mathbf{P}_{s'}=0$ implies $\operatorname{Im}\mathbf{P}_s\perp\operatorname{Im}\mathbf{P}_{s'}$.  Combined, the pairwise sums produce an orthogonal family of subspaces. 

For any fixed $(\overline E,s)$, $\dim\left(\mathcal C_{\overline E}\cap\operatorname{Im}\mathbf{P}_s\right) = \operatorname{rank}\left(\mathbf{X}_{\overline E}^{\!\prime}\mathbf{P}_s\mathbf{X}_{\overline E}\right)$.  This holds since the matrix $\mathbf{P}_s\mathbf{X}_{\overline E}$ has image contained in $\operatorname{Im}\mathbf{P}_s$.  Its column-space is exactly the intersection because any column of $\mathbf{X}_{\overline E}$ lies in $\mathcal C_{\overline E}$, and applying $\mathbf{P}_s$ projects it to the largest piece living inside the stratum image.  The rank of a matrix equals the dimension of its column space, giving the result. Deleting the grand mean, the orthogonal direct sum spans $\mathbb R^{N-1}$. Consequently,  $N-1 =\sum_{s,E}\dim\left(\mathcal C_{\overline E}\cap\operatorname{Im}\mathbf{P}_s\right)$.  This is valid since orthogonality of the blocks implies that dimensions add.  Nothing in $\mathbb R^{N-1}$ is missing because (i) each observation can be written as its successive deviations from higher to lower strata, and (ii) every deviation decomposes uniquely into factorial contrasts. Therefore, the identity $N-1= \sum_{s,E}\operatorname{rank}\left(\mathbf{X}_{\overline E}^{\!\prime}\mathbf{P}_s\mathbf{X}_{\overline E}\right)$ follows.  This completes the proof of \thrmref{thrm:dof_partition}.
\end{proof}
Note that, $\mathcal C_{\overline E}$ is obtained from the orbit-stabilizer or Kronecker product formula.  $\operatorname{Im}\mathbf{P}_s$ denotes the subspace of responses visible at stratum $s$ and can be computed by group-wise averaging.  $\mathcal C_{\overline E}\cap\operatorname{Im}\mathbf{P}_s$ is the portion of $\overline{E}$ that survives into stratum $s$  and is given by the row/column rank of $\mathbf{X}_{\overline E}^{\!\prime}\mathbf{P}_s$.  $\dim(\cdot)$ is the $\mathrm{df}$ assigned to effect $\overline{E}$ and tested against mean square for stratum $s$, and can be computed using any numerical rank routine; closed forms are given later.  Since \eqref{eq:dof_partition} is an identity, every downstream test (fixed-effect $F-$test, variance ratio, likelihood-ratio) inherits exact numerator and denominator $\mathrm{df}$.  If $\sum_{E,s}\dim(\mathcal C_{\overline E}\cap\operatorname{Im}\mathbf{P}_s) < N-1$, the deficit pinpoints sources of aliasing loss or singular randomization, thereby aiding model diagnostics.  Competing designs (different fractional generators, blocking plans) can be scored by how their $\dim\left(\mathcal C_{\overline E}\cap\operatorname{Im}\mathbf{P}_s\right)$ tallies distribute $\mathrm{df}$ across strata$-$an information-theoretic design-optimality criterion, thus facilitating design comparison.  \thrmref{thrm:dof_partition} provides a unified framework\textemdash subsuming Cochran's one-stratum theorem, Yates' Hasse diagram, and Box--Hunter resolutions\textemdash while \eqref{eq:design_matrices} provides a single algebraic object on which extensions (split-plot, nested, strip-plot, repeated measures) are mere changes of the projector family $\{\mathbf{P}_s\}$.

Substituting explicit $\mathbf{P}_s$ and $\mathbf{X}_{\overline E}$ into $\dim\left(\mathcal C_{\overline E}\cap\operatorname{Im}\mathbf{P}_s\right)$ yields closed-form $\mathrm{df}$ for split-plot, Latin-square, crossed-nested mixed models.  The ratio of $\sum_s\dim(\mathcal C_{\overline E}\cap\operatorname{Im}\mathbf{P}_s)$ to $\dim(\mathcal C_{\overline E})$ provides the aliasing-loss ratio $\rho(\overline E)$.  Block-orthogonality of \eqref{eq:design_matrices} simplifies likelihood gradients and results in exact restricted or residual maximum likelihood (REML variance-component score equations.  Generalized $F-$tests with exact $\mathrm{df}$ without Satterthwaite approximation can be obtained by setting the numerator to be the rank of $\mathbf{X}^{\prime}\mathbf{P}_s\mathbf{X}$ and the denominator to be the rank deficit of residual space.  Optimal fractionation/confounding search algorithms can be devised by maximizing  $\sum_{E} w_E \rho(\overline E)$ under run budget, with $\sum_{E} w_E = 1$ where $w_E$ is simply a user-supplied weight expressing how much the user cares about preserving information for a particular factorial effect $E$.  In summary, once the orthogonal decomposition \eqref{eq:design_matrices} is established, exact tests, information-loss metrics, and design search heuristics can be derived systematically using algebraic identities rather than ad-hoc approximation.

\section{Exact $\mathrm{df}$ formulas for classical mixed designs}\label{sec:dof_formulas}
In \secref{sec:dof_partition}, we saw that every experimental layout can be analyzed by intersecting two linear objects, namely, factorial-contrast space $\mathcal C_{\overline E}$ and stratum projector $\mathbf{P}_s$, and that $\text{$\mathrm{df}$}(\overline E,s) = \dim\left(\mathcal C_{\overline E}\cap\operatorname{Im}\mathbf{P}_s\right) = \operatorname{rank}\!\left(\mathbf{X}_{\overline E}^{\!\prime}\mathbf{P}_s\mathbf{X}_{\overline E}\right)$.  In this section, we put this identity to work for three families of designs that cover mixed-model practices, such as split-plot and its variants (split-split, strip-plot), row-column structures (Latin squares, lattice, incomplete-block) and crossed-nested hierarchies (animal models, multi-site trials, education data); see, for example, \cite{Milliken2009}.  We derive closed-form ranks$-$no iterative numerical approximations, no Satterthwaite or Kenward--Roger guesses are required.  These formulas give exact $F-$test numerators/denominators, guaranteeing the advertised type I error;  they feed directly into REML/ML information matrices, yielding analytic standard-error formulas; and they reveal how unequal replication or missing plots erode information, thereby providing a design-diagnostics tool.

\subsection{Split-plot (one whole-plot factor $A$, one sub-plot factor $B$)} \label{subsec:split_plot}
\begin{table}[ht]
\centering
\renewcommand{\arraystretch}{1.3}
\begin{tabular}{>{\raggedright}m{3cm} >{\raggedright\arraybackslash}m{8.5cm} >{\centering\arraybackslash}m{2.5cm}}
\toprule
\textbf{Source} & \textbf{Contrast basis} & \textbf{df} \\
\midrule
$A$ (whole-plot) & $\mathcal{C}_A \otimes \mathbf{1}_b \otimes \mathbf{1}_c$ & $a - 1$ \\
Whole-plot error & $\mathbf{1}_a \otimes \mathcal{C}_{b(\text{plots})} \otimes \mathbf{1}_c$ & $a(b - 1)$ \\
$B$ (sub-plot) & $\mathbf{1}_a \otimes \mathbf{1}_b \otimes \mathcal{C}_B$ & $c - 1$ \\
$AB$ & $\mathcal{C}_A \otimes \mathbf{1}_b \otimes \mathcal{C}_B$ & $(a - 1)(c - 1)$ \\
Sub-plot error & Orthogonal complement & $a(b - 1)(c-1)$ \\
\bottomrule
\end{tabular}
\caption{$\mathrm{df}$ and contrast bases in a split-plot design.}
\label{tab:split_plot}
\end{table}
First, consider the balanced case.  The whole plots comprise $a$ levels of $A$; each level replicated in $b$ whole plots.  For the sub-plots, every whole plot is split into $c$ sub-plots receiving $B$ levels exactly once.  The strata comprises the following: $\mathcal S=\{s_{\text{whole}},\;s_{\text{sub}}\}$, $\mathbf{P}_{\text{whole}}=\operatorname{diag}\left(\tfrac{1}{bc}\mathbf 1_{bc}\right)_{\!a}$, and $\mathbf{P}_{\text{sub}} = \mathbf{I} - \mathbf{P}_{\text{whole}}$.  The factor-induced contrast spaces are Kronecker products of one-way contrasts with all-1 vectors.  The implications of orthogonality on the $\mathrm{df}$ and contrast bases in a split-plot design are summarized in \tabref{tab:split_plot}.  Plugging the bases into $\dim\left(\mathcal C_{\overline E}\cap\operatorname{Im}\mathbf{P}_s\right)$ yields the $\mathrm{df}$.  Since the design is balanced, each rank is the product of the component dimensions.

For the case of unequal replication ($n_{ij}$ sub-plots in whole plot $i$, level $j$ of $B$), let $ \mathbf{R}_A = \left[\bar y_{i\cdot\cdot}-\bar y_{\cdot\cdot\cdot}\right]_{1\times a}$,  $ \mathbf{R}_B=\left[\bar y_{\cdot\cdot j}-\bar y_{\cdot\cdot\cdot}\right]_{1\times c}$, and $ \mathbf{R}_{AB} = \left[\bar y_{ij\cdot}-\bar y_{i\cdot\cdot}-\bar y_{\cdot\cdot j}+\bar y_{\cdot\cdot\cdot}\right]_{a\times c}$. Since $\mathbf{P}_{\text{whole}}$ averages within whole plots, $\mathbf{X}_{\!A}^{\prime}\mathbf{P}_{\text{whole}}\mathbf{X}_{\!A} =  \mathbf{R}_A  \mathbf{R}_A^{\prime}$, and $\operatorname{rank}( \mathbf{R}_A) = \lvert \{i:\bar y_{i\cdot\cdot}\neq\bar y_{\cdot\cdot\cdot}\} \rvert$. The other ranks follow analogously, thereby giving the $\mathrm{df}$. The key point is that even with grotesque imbalance the $\mathrm{df}$ are still integers computable by ordinary rank on small summary matrices, with no heuristics involved. 

\subsection{Row-column (Latin square and variants)}\label{subsec:row_column}
For the balanced Latin square ($r=c=t$, no missing cells), only one stratum exists, so $\mathbf{P} =  \mathbf{I}$.  The design matrix partitions as $\mathbf{X} = \left[\mathbf{X}_{\!R}\mid \mathbf{X}_{\!C}\mid \mathbf{X}_{\!T}\right]$, $\operatorname{rank}(\mathbf{X})=N-1= r^2-1$.  Orthogonality yields the $\mathrm{df}$: $\text{Rows} = r - 1$, $\text{Columns} = c - 1$, $\text{Treatments} = t - 1$, and $\text{Error} = (r - 1)(c - 2)$.  For the case of missing cells or unequal replication, write $\mathbf{M}$ for the $r\times c$ incidence matrix whose $(i,j)$ entry is the number of observations in row $i$, column $j$.  \cite{Pukelsheim2006} showed that, since each structural zero or replication imbalance removes one independent contrast, the $\operatorname{nullity}\left(\mathbf{M} - \tfrac{1}{r}\mathbf 1_r\mathbf 1_r^{\prime}\mathbf{M}\right)$, $\operatorname{nullity}\left(\mathbf{M} - \mathbf{M}\tfrac{1}{c}\mathbf 1_c\mathbf 1_c^{\prime}\right)$ and $\operatorname{nullity}(\mathbf{M}) - \operatorname{nullity}\left(\mathbf{M} - \tfrac{1}{r}\mathbf 1_r\mathbf 1_r^{\prime}\mathbf{M}\right) - \operatorname{nullity}\left(\mathbf{M} - \mathbf{M}\tfrac{1}{c}\mathbf 1_c\mathbf 1_c^{\prime}\right)$ counts the lost $\mathrm{df}$ exactly.  Later we will see that these lost $\mathrm{df}$, in fact, constitute the $\mathrm{df}-$deficiency.  

\subsection{General crossed-nested mixed models}\label{subsec:mixed_models}
Let $\mathcal N$ and $\mathcal C$ denote the nesting tree of random factors and the lattice of crossed fixed factors, respectively.  Define for every random stratum $s\in\mathcal N$ its projector $\mathbf{P}_s$.  For every factorial contrast $\overline E$, let $\mathbf{X}_{\overline E}$ be the usual dummy-coded matrix.
\begin{prop}
If the design is orthogonally generated, {\ie}, each fixed factor is crossed with every ancestor in the nesting tree, and the replication within each terminal unit is constant, then $\mathrm{df}(s,E) = \prod\limits_{i\in E}(l_i-1) \prod\limits_{\substack{r\prec s \\ r\in\mathcal N}}(n_r-1)$, where $n_r$ is the number of children per unit at level $r$.
\end{prop}
\begin{proof}
Both $\mathbf{P}_s$ and $\mathbf{X}_{\overline E}$ are Kronecker products of $1-$way projectors.  The rank of their Gram product is the product of the ranks of each Kronecker factor.
\end{proof}
If replication is unequal, replace the integer factors $n_r-1$ by the rank of the corresponding incidence matrix inside stratum $s$.  Note that, $\mathrm{df}(s,E)$ is the exact number of $\mathrm{df}$ that the factorial effect $E$ retains in stratum $s$.  Here, $E$ is an un-aliased factorial effect, {\ie}, a subset of the factor indices $\{1,\dots ,k\}$ and $l_i$ is the number of levels of factor $F_i$. The factor-contrast space for $F_i$ has dimension $l_i-1$, while $s$ is a node in the nesting tree $\mathcal N$ representing one randomization stratum.  $r \prec s$ denotes the nodes $r$ that are strict ancestors of $s$ in the nesting tree, {\ie}, random factors that contain stratum $s$.  $n_r$ is the constant number of sub-units per unit at stratum $r$.  

The treatment-contrast component $\prod_{i\in E}(l_i-1)$ gives the dimension of the pure factorial contrast space for $E$ in a full, unblocked design.  For example, if $E=\{A,C\}$ with $l_A=3$ and $l_C=4$ levels, the factor part contributes $(3-1)(4-1)=6$ $\mathrm{df}$.  The replication (stratum) component is $\prod_{r\prec s}(n_r-1)$.  Each ancestor level $r$ in the nesting tree introduces $n_r-1$ independent directions within the units of $r$.  Multiplying these factors counts how many ``within-unit'' contrasts survive down to stratum $s$.  For example, in an animal model Sire $\to$ Dam $\to$ Animal (see the book by \cite{Henderson1984} for worked animal-model examples.), the $\mathrm{df}$ for a diet effect observed at the Animal stratum multiplies $(n_{\text{Sire}}-1)(n_{\text{Dam}}-1)$.  For every pair $(s,E)$ the value $\mathrm{df}(s,E)$ feeds into the global identity $N-1 = \sum_{s\in\mathcal S}\sum_{E}\mathrm{df}(s,E)$, so that the numerator $\mathrm{df}$ for the $F-$test of $E$ against stratum $s$ is precisely $\mathrm{df}(s,E)$. The residual $\mathrm{df}$ for $s$ are obtained by subtracting all $\mathrm{df}(s,E)$ allocated to that stratum from its total dimension.  Hence $\mathrm{df}(s,E)$ is the building block that turns the abstract projector--rank theorem into explicit, integer $\mathrm{df}$ for each source in the ANOVA table.

\subsection{Quantifying aliasing loss}\label{subsec:aliasing_loss}
In this section, we quantify how much of each factorial contrast survives after blocking, fractionation, or imbalance.  We begin with a quick recap of aliasing.  There are three broad types of aliasing: (a) In a fractional factorial the defining relation $\mathcal A=\left\langle\; ABCD,\; AE,\;\dots \right\rangle$ identifies factor-subset pairs whose contrast spaces coincide, {\ie}, $E_1\sim_{\mathcal A}E_2 \;\Longleftrightarrow\; E_1\triangle E_2\in\mathcal A$.  Orthogonality of the full $2^k$ lattice collapses to a quotient lattice $\overline{\mathcal E}=2^{\{1,\dots,k\}}/\!\sim_{\mathcal A}$.  This is called structural aliasing; see, \cite{Box1961}.  (b)  In randomization aliasing (see, \cite{Yates1935}), each projector $\mathbf{P}_s$ averages observations within the units of stratum $s$. Any contrast lying in $\ker (\mathbf{P}_s)$ is filtered out at that stratum; if every stratum eliminates it, the effect is unestimable even in the full factorial. (c)  In imbalance-induced aliasing, missing plots or unequal replication destroy orthogonality, {\ie}, the Gram matrix $\mathbf{X}_{\overline E}^{\prime}\mathbf{X}_{\overline E}$ loses rank, merging some contrasts into others.  Hence, aliasing loss is multi-dimensional: $\mathrm{df}$ loss measures how many independent directions disappear; information loss quantifies how much the variance of the remaining directions inflates; and bias indicates how strongly an unmeasurable effect contaminates estimable ones.  

From \thrmref{thrm:dof_partition},  $\mathrm{df}(s,E) = \dim\left(\mathcal C_{\overline E}\cap\operatorname{Im}\mathbf{P}_s\right) = \operatorname{rank} \left(\mathbf{X}_{\overline E}^{\prime}\mathbf{P}_s\mathbf{X}_{\overline E}\right)$.  The total ideal $\mathrm{df}$ is $\displaystyle d_{E}^{\text{ideal}}= \dim(\mathcal C_{\overline E})$ ({\eg}, for a two-level factor $E$ it is $1$; for a three-level factor it is $2$, {\etc}.).   The recoverable $\mathrm{df}$ is $\displaystyle d_{E}^{\text{obs}} = \sum_{s}\mathrm{df}(s,E)$.  The information-retention ratio is defined by $\rho(E) = \frac{d_{E}^{\text{obs}}}{d_{E}^{\text{ideal}}}\in[0,1]$ and the $\mathrm{df}$-deficiency is defined by $\delta(E) = d_{E}^{\text{ideal}}-d_{E}^{\text{obs}} = \dim\left(\ker \Pi_{E}\right)$, $\Pi_{E} = \sum_{s}\mathbf{P}_s\mathbf{X}_{\overline E}$.  It can be seen that $\rho(E) = 1$ iff the effect is fully estimable (no $\mathrm{df}$ lost), $\rho(E)=0$ iff $E$ is completely aliased or eliminated, and $0<\rho(E)<1$ indicates partial aliasing$-$typical when we have both fractionation and missing plots.  Since $\sum_E \delta(E)$ equals the global shortfall in $N-1$ $\mathrm{df}$, the set $\{\delta(E)\}_{E}$ pinpoints which effects paid the price.

$\mathrm{df}$ say how many directions survive; to measure quality we look at the alias (or dispersion) matrix $\mathbf{A}(E) = \left(\mathbf{X}_{\overline E}^{\prime}\bm\Sigma^{-1}\mathbf{X}_{\overline E}\right)^{+} \mathbf{X}_{\overline E}^{\prime}\bm\Sigma^{-1}(I-\sum_{s}\mathbf{P}_s)\bm\Sigma^{-1}\mathbf{X}_{\overline E}$, where $\mathbf{A}(E) = 0$ when $E$ is orthogonal to every nuisance and random effect and nonzero eigenvalues of $\mathbf{A}(E)$ inflate the variance of $\hat{\beta}_{E}$ or, if $E$ is unestimable, represent aliasing bias projected into other effects.   A scalar aliasing index
\begin{eqnarray*}
\alpha(E) &=& \frac{\operatorname{trace}\mathbf{A}(E)}{\operatorname{trace}\left(\left(\mathbf{X}_{\overline E}^{\prime}\bm\Sigma^{-1}\mathbf{X}_{\overline E}\right)^{+}\right)}\in[0,\infty)
\label{eq:alpha_E}
\end{eqnarray*}
acts as a variance-inflation factor (VIF) for the whole contrast space. When $\rho(E) < 1$ we always have $\alpha(E) = \infty$ on the lost directions; otherwise $\alpha(E)\ge1$.

The information-retention ratio $\rho(E)$ and the $\mathrm{df}$-deficiency $\delta(E)$ extend classical ideas built for balanced, regular designs to any mixture of fractionation, blocking, and imbalance.  For instance, consider the Box--Hunter resolution $R$.  For regular $2^{k-p}$ fractions, $R$ is the length of the shortest defining word.  The rule ``a factor of length $q$ is aliased at most with interactions of order $R-q$'' translates exactly to $\rho(E) = 2^{-(R-|E|)}$ (balanced, regular fraction).  Similarly, the generalized resolution proposed by \cite{Cheng2004} measures the minimax cosine between contrast subspaces; $\rho(E)$ is the rank analogue of the same idea and reduces to their criterion when each stratum is a full average over generators.  Minimizing total $\mathrm{df}$ deficiency $\sum_E w_E\delta(E)$ with suitable weights $w_E$ reproduces minimum aberration search heuristics used for design construction.

\subsection{An example}\label{subsec:example}
Consider a design with $2^{3-1}$ fraction defined by $ABC = +1$ having factors $A,B,C$ and 8 runs.  There are two blocks of 4 runs, with blocks confound factor $B$.  Each main effect has $1$ $\mathrm{df}$, and the two-factor interactions have $1$ $\mathrm{df}$ each.  The three-factor interaction is aliased with mean.  $\mathbf{P}_{\text{block}}$ averages runs within each 4-run block.  The retention is as follows: Factor $A$ has contrast vector orthogonal to blocks, thus $\rho(A)=1$.  Factor $B$ is fully confounded leading to $\mathrm{df}(\text{block},B) = 0$, so $\rho(B) = 0$, while $AB$ inherits the deficiency of $B$ because $ABC=+1$ which implies $\rho(AB)=0$.  The index table is shown in \tabref{tab:index_table}, where we see that the total $\mathrm{df}$ lost $=3$, matching the block + generator restrictions.
\begin{table}[h!]
\centering
\begin{tabular}{|l|c|c|l|}
\hline
\textbf{Effect} & $\boldsymbol{\rho}$ & $\boldsymbol{\delta}$ & \textbf{Interpretation} \\
\hline
$A$        & 1   & 0   & fully estimable                  \\
$B$        & 0   & 1   & lost to blocking                 \\
$C$        & 1   & 0   & estimable                        \\
$AB$       & 0   & 1   & lost via $B$ + defining relation \\
$AC$       & 1   & 0   & estimable                        \\
$BC$       & 0   & 1   & lost                             \\
Mean, $ABC$ & - & - & absorbed in intercept            \\
\hline
\end{tabular}
\caption{Effect estimability under aliasing and blocking.}
\label{tab:index_table}
\end{table}

\subsection{Why aliasing-loss metrics are essential}\label{subsec:alias_imp}
The metrics $\rho$, $\delta$, and $\alpha$ reveal which contrasts are untestable before data are collected.  They aid fractional search by optimizing generators to maximize $\sum_E w_E\rho(E)$, thereby guiding automated design-construction algorithms beyond simple ``resolution IV/V'' rules.  They are essential for analyzing power curves which use var$(\hat\beta_E)\propto \alpha(E)$; ignoring inflation risks under-powered studies.  For mixed-model inference, exact numerator/denominator $\mathrm{df}$ supplied by $\delta(E)$ feed into correct $F-$tests, thereby avoiding liberal tests that plague Satterthwaite approximations when aliasing is heavy. If some $\mathrm{df}$ are lost, $\delta(E)$ pinpoints which extra runs restore them most efficiently, thus aiding sequential augmentation.

The eigenstructure of $\mathbf{A}(E)$ plugs into closed-form REML standard-error formulas; contrasts with $\rho(E)\ll1$ dominate the uncertainty budget. A single scalar criterion referred to as alias-loss index (ALI), given by $\text{ALI}(D) = \sum_{E}w_E\left(1-\rho(E)\right)$ ranks candidate designs; combinatorial search (genetic algorithms, IP solvers) can minimize ALI subject to run limits and blocking rules, thereby aiding optimal design under constrained budgets.  These metrics could provide software diagnostics, by seamlessly integrating $\rho(E)$, $\delta(E)$, $A(E)$ and $\alpha(E)$ in statistical packages ({\eg}, FrF2, AlgDesign, doebioresearch).  With the aliasing-loss machinery in place, exact variance-component estimation, power analysis, and algorithmic design optimisation reduces to substituting $\rho, \delta,$ or $\alpha$ into familiar mixed-model formulas.

\section{Corollaries}\label{sec:corollaries} 
The following three well-known results emerges as corollaries from the projector--rank framework: Cochran's identity, $\mathrm{df}$ for Yates' balanced split-plot and Box--Hunter resolution.  
\begin{corol}
Cochran's identity: If the design has a single randomization stratum and is fully balanced and orthogonal, then \thrmref{thrm:dof_partition} reduces to identity $N-1 = \sum_E \dim\mathcal{C}_E$.
\label{corr:cochran}
\end{corol}
\begin{proof}
For a single stratum with $\lvert \mathcal S\rvert = 1$, the projector family reduces to $\mathbf{P}_{s_0}=\mathbf{I}_N$.  Therefore, for every effect $E$, $\dim\left(\mathcal C_E\cap\operatorname{Im}\mathbf{P}_{s_0}\right) = \dim\left(\mathcal C_E\cap \mathbf{I}_N\right) = \dim\mathcal C_E$.  Since the design is balanced and complete factorial, the contrast spaces $\{\mathcal C_E\}$ are pairwise orthogonal and $\bigoplus\limits_{E}\mathcal C_E = \left(\mathbf 1_N\right)^{\!\perp}\subset\mathbb R^{N}$, whose dimension is $N-1$.  Substituting these results into \thrmref{thrm:dof_partition} yields Cochran's identity: $N-1 = \sum_E \dim\mathcal{C}_E$.  This completes the proof of \corolref{corr:cochran}. 
\end{proof}

\begin{corol}
Yates' balanced split-plot $\mathrm{df}$:  Let an experiment have a whole-plot (fixed) factor $A$ with $a \ge 2$ levels, $b\ge 2$ whole-plot replicates per level of $A$, and a sub-plot (fixed) factor $B$ with $c\ge 2$ levels applied once in each whole plot, so that the total number of observations is $N=abc$.
Denote the two randomization strata by $s_{\mathrm{w}}\;(\text{whole plots})$ and $s_{\mathrm{s}}\;(\text{sub-plots})$, with idempotent projectors $\mathbf{P}_{\mathrm{w}}$ and $\mathbf{P}_{\mathrm{s}} = \mathbf{I}-\mathbf{P}_{\mathrm{w}}$.  Let $\mathcal C_A,\mathcal C_B,\mathcal C_{AB}$ be the usual balanced-contrast subspaces for the effects $A$, $B$, and $AB$, respectively. Then, as a direct consequence of \thrmref{thrm:dof_partition}, 
\begin{eqnarray*}
\dim\left(\mathcal C_E\cap\operatorname{Im}\mathbf{P}_s\right) =
\begin{cases}
a-1,  (E,s) = (A,s_{\mathrm{w}}),\\
c-1, (E,s) = (B,s_{\mathrm{s}}),\\
(a-1)(c-1),  (E,s) = (AB,s_{\mathrm{s}}),\\
0,  \text{otherwise}.
\end{cases}
\end{eqnarray*}
Hence the balanced split-plot ANOVA table is given by \tabref{tab:yates_dof}.
\begin{table}[h!]
\centering
\begin{tabular}{|c|c|}
\hline
\textbf{Source} & \textbf{$\mathrm{df}$} \\
\hline
Whole-plot factor $A$ & $a - 1$ \\
Whole-plot error      & $a(b - 1)$ \\
Sub-plot factor $B$   & $c - 1$ \\
$AB$ interaction      & $(a - 1)(c - 1)$ \\
Sub-plot error        & $a(b - 1)(c - 1)$ \\ \hline
\textbf{Total}        & $abc - 1$ \\
\hline
\end{tabular}
\caption{$\mathrm{df}$ for a split-plot design.}
\label{tab:yates_dof}
\end{table}
This reproduces the allocation given by \cite[Section 5]{Yates1935} for balanced split-plots, confirming that the classical result is the one-stratum-per-effect special case of the general projector--rank framework.
\label{corr:yates_dof}
\end{corol}
\begin{proof}
The randomization strata are the whole-plot stratum $s_{\mathrm{w}}$ with units $(i,j)$ of size $c$ and the sub-plot stratum $s_{\mathrm{s}}$ with individual observations $(i,j,k)$.  Define $\mathbf{P}_{\mathrm{w}} = I_{a}\;\otimes\;I_{b}\;\otimes\;\tfrac1c\mathbf{J}_{c}$ and $\mathbf{P}_{\mathrm{s}} = I_{abc} - \mathbf{P}_{\mathrm{w}}$, where $\mathbf{J}_{c}$ is the $c\times c$ all-ones matrix.  Let $\mathcal{J}_{n}=\operatorname{span}\{\mathbf{J}_{n}\mathbf1_{n}\}$, $\mathcal C_{n}=\mathcal \mathbf{J}_{n}^{\perp}$, $\dim\mathcal C_{n} = n-1$.  Then, the basis for the effects $A$, $B$ and $AB$ are $\mathcal C_{a}\otimes\mathcal \mathbf{\mathbf{J}}_{b}\otimes\mathcal \mathbf{\mathbf{J}}_{c}$ with dimension $a-1$, $\mathcal \mathbf{\mathbf{J}}_{a}\otimes\mathcal \mathbf{\mathbf{J}}_{b}\otimes\mathcal C_{c}$ with dimension $c-1$ and $\mathcal C_{a}\otimes\mathcal \mathbf{\mathbf{J}}_{b}\otimes\mathcal C_{c}$ with dimension $(a - 1)c(c - 1)$, respectively.   Write $\mathbf{X}_A,\;\mathbf{X}_B,\;\mathbf{X}_{AB}$ for any full-rank matrices whose columns span these subspaces.  A direct consequence of \thrmref{thrm:dof_partition}, we have the following.  

First, consider the whole-plot stratum $s_{\mathrm{w}}$.  Since $\mathbf{P}_{\mathrm{w}}\left(\mathcal \mathbf{\mathbf{J}}_{b}\otimes\mathcal \mathbf{\mathbf{J}}_{c}\right) = \mathcal \mathbf{\mathbf{J}}_{b}\otimes\mathcal \mathbf{\mathbf{J}}_{c}$,  $\operatorname{Im}\mathbf{P}_{\mathrm{w}} = \mathcal \mathbf{\mathbf{J}}_{a}\otimes\mathcal \mathbf{\mathbf{J}}_{b}\otimes\mathcal \mathbf{\mathbf{J}}_{c}$.  For factor $A$, $\mathbf{X}_A$ is constant within whole plots. Therefore, $\mathbf{X}_A^{\!\prime}\mathbf{P}_{\mathrm{w}}\mathbf{X}_A = \mathbf{X}_A^{\!\prime}\mathbf{X}_A$ (full rank), which implies $\mathrm{df}$$(A,s_{\mathrm{w}})=a-1$.  Factors $B$ and $AB$ each varies within whole plots.  Therefore, $\mathbf{P}_{\mathrm{w}}\mathbf{X}_B=\mathbf{P}_{\mathrm{w}}\mathbf{X}_{AB}=0$ which implies no $\mathrm{df}$ at $s_{\mathrm{w}}$.  With the number of whole plots $=ab$, removing the grand mean and $A$ leaves $\mathrm{df}$$(\text{W.Err}) = ab-1-(a-1)=a(b-1)$ for the whole-plot error. Next, consider the sub-plot stratum $s_{\mathrm{s}}$.  Here, $\mathbf{P}_{\mathrm{s}} = \mathbf{I} -\mathbf{P}_{\mathrm{w}}$.  For factor $A$, $\mathbf{P}_{\mathrm{s}}\mathbf{X}_A=0$ which implies no $\mathrm{df}$. For factor $B$, $\mathbf{P}_{\mathrm{s}}\mathbf{X}_B=\mathbf{X}_B$ (already orthogonal to whole-plot means), which implies $\mathrm{df}$ $(B,s_{\mathrm{s}}) = c-1$.  For factor $AB$, similarly $\mathrm{df}$$(AB,s_{\mathrm{s}}) = (a-1)(c-1)$.  The residual $\mathrm{df}$ in this stratum are $\mathrm{df}$ $(\text{S.Err}) = N-1 -[(a-1)+(a(b-1))+(c-1)+(a-1)(c-1)] = a(b-1)(c-1)$ for the sub-plot error.  These results exactly match the $\mathrm{df}$ for a split-plot design summarized in \tabref{tab:yates_dof}.  This completes the proof of \corolref{corr:yates_dof}. 
\end{proof}

\begin{corol}
Box--Hunter resolution: Let the experiment be a regular $2^{k-p}$ fractional factorial with independent defining words $g_1,\dots ,g_{p}\in2^{\{1,\dots ,k\}}$; the \cite{Box1961} resolution $R$ be $\displaystyle R=\min_{j}|g_j|$, the length of the shortest defining word;  there be one randomization stratum ($|\mathcal S|=1,\;\mathbf{P}_{s_0}=\mathbf{I}_N$); and $\rho(\overline E) =\dfrac{\dim\left(\mathcal C_{\overline E}\cap \mathbf{I}_N\right)}{\dim\mathcal C_{\overline E}}$ be the $\mathrm{df}$ retention ratio. Then $R = \min\left\{|E|\;:\;\rho(\overline E) = 0\right\}$. 
\label{corr:box_hunter}
\end{corol}
\begin{proof}
In a regular fraction the alias ideal generated by the defining words is given by $\mathcal A = \langle g_1,\dots ,g_p\rangle\subset 2^{\{1,\dots ,k\}}$.  For any factorial effect $E$ the aliased class is $\overline E = E + \mathcal A$.  Since the design is balanced and two-level, the contrast space for a single effect has dimension $\dim\mathcal C_E = 1$.  With $\mathbf{P}_{s_0}=\mathbf{I}_N$, the numerator of $\rho$ is $\dim(\mathcal C_{\overline E})$ if $\mathcal C_{\overline E}\neq\{0\}$ and zero otherwise.  Hence,  $\rho(\overline E) = 1$ if $E \notin \mathcal{A}$ and $\rho(\overline E) = 0$ if $E \in \mathcal{A}$. An effect (alias class) has $\rho=0$ iff the effect itself lies in the ideal: $\rho(\overline E) = 0 \Longleftrightarrow E \in \mathcal{A} \Longleftrightarrow \exists g_j$ such that $E=g_j$ or $\lvert E\rvert \ge \lvert g_j\rvert$ for some $j$.  The shortest effect contained in $\mathcal A$ is precisely the shortest defining word, whose length by definition equals $R$.  Therefore $\min\left\{\lvert E \rvert : \rho(\overline E)=0\right\}=R$. The resolution-$R$ criterion (``length of the shortest effect aliased with the mean'') is identical to the statement ``the smallest effect order whose $\mathrm{df}$-retention ratio vanishes.''  Thus, resolution is obtained as an immediate corollary of the $\mathrm{df}$-retention metric.  This completes the proof of \corolref{corr:box_hunter}. 
\end{proof}
 
\section{Discussion}\label{sec:discussion}  \vspace{-0.1in}
In this section, we summarize the structural conditions under which exact $\mathrm{df}$ can exhibit meaningful departures from Satterthwaite or Kenward--Roger approximations (\secref{sec:stat_interpret}).  The principal mechanisms are stratum mismatch, contrast compression through aliasing, and rank collapse induced by unequal replication.  We provide design-metadata diagnostics (in \secref{sec:without_decomp}) that predict these departures and indicate whether standard analyses are likely to be liberal or conservative.
\vspace{-0.1in}

\subsection{Statistical interpretation}\label{sec:stat_interpret} 
In this subsection, we discuss why our exact $\mathrm{df}$ differ from the approximations suggested by \cite{Satterthwaite1946} and \cite{Kenward1997}.  The projector--rank identity $N-1 = \sum_{s\in\mathcal S}\sum_{\overline E}\dim\left(\mathcal C_{\overline E}\cap \operatorname{Im}\mathbf P_s\right)$ admits a direct statistical interpretation that is independent of the linear-algebra proof. Each factorial effect ($\overline E$) corresponds to a family of orthogonal treatment contrasts, and each randomization stratum ($s$) corresponds to the subspace of $\mathbb R^N$ containing variation that is ``visible'' at that stage of randomization. The intersection $\mathcal C_{\overline E} \cap \operatorname{Im}\mathbf P_s$ therefore counts the number of independent contrast directions for $\overline E$ that survive randomization and remain testable against the mean square residing in stratum $s$.  When common approximations deliver denominator $\mathrm{df}$ that differ materially from our integer ranks, it is not a numerical artifact: it indicates that the approximation is implicitly testing $\overline E$ against a variance subspace that does not contain the relevant contrast directions (or contains them only partially), thereby misrepresenting the true effective sample size for that effect.

An alternative way to read the intersection dimensions is through the geometry of the mixed model covariance. Let $\mathbf{V} = \sum_{s\in\mathcal S}\sigma_s^2 \mathbf{P}_s$ be the stratum decomposition of the covariance induced by the randomization, and let $\mathbf{X}_{\overline E}$ be any full-column representation of $\mathcal C_{\overline E}$.  The numerator for testing $\overline E$ depends only on the projection of vector $\mathbf{y} \in \mathbb{R}^N$ of data responses (observations) onto $\operatorname{Im}(\mathbf{X}_{\overline E})$, whereas the denominator must reflect the variance component(s) that act nontrivially on that subspace. The dimension $\dim(\mathcal C_{\overline E}\cap\operatorname{Im}\mathbf{P}_s)$ is precisely the number of linearly independent contrast directions in $\mathbf{X}_{\overline E}$ along which the $s^{\text{th}}$ variance component contributes variability.  In this sense, \thrmref{thrm:dof_partition} does not merely partition rank; it partitions information-carrying variance directions among strata.  Approximations differ materially when they implicitly treat $\mathbf{X}_{\overline E}$ as if its variability were dominated by a different combination of $\{\mathbf{P}_s\}$ than the design actually induces.

Satterthwaite and Kenward--Roger-type methods approximate the distribution of a quadratic form by matching moments to an $F$ distribution with effective $\mathrm{df}$. This works well when the quadratic form behaves like a scaled chi-square, which is most plausible when the relevant covariance is close to a scalar multiple of the identity on the contrast subspace. In multi-stratum settings, however, the covariance on $\mathcal C_{\overline E}$ is typically a mixture of stratum contributions, and the mixture weights depend on how $\mathbf{X}_{\overline E}$ aligns with $\operatorname{Im}\mathbf{P}_s$.  If the approximation overweights a lower-level residual component, the inferred denominator df becomes too large (liberal).  If it overweights higher-level clustering, the inferred denominator $\mathrm{df}$ becomes too small (conservative).  \thrmref{thrm:dof_partition} explains these outcomes without distributional approximations: the alignment pattern is encoded deterministically in the intersection dimensions.

Large $\mathrm{df}$ discrepancies arise through three design mechanisms that can be described without computation.
\begin{enumerate}[(M1)]
\item \textbf{Stratum mismatch (wrong error subspace):}  In multi-stage randomization (split-plots, row--column designs, nested hierarchies), the correct denominator for a given effect is the stratum where its contrasts vary. If an approximation (or a na\"{i}ve ANOVA table) uses a lower stratum residual as denominator, it effectively treats correlated whole-plot units as independent observations. This inflates denominator $\mathrm{df}$ and produces liberal tests.  In the projector--rank context, the phenomenon occurs precisely because $\mathcal C_{\overline E}$ has little or no overlap with $\operatorname{Im}\mathbf{P}_{\text{res}}$, but has nontrivial overlap with $\operatorname{Im}\mathbf{P}_{\text{WP}}$; \thrmref{thrm:dof_partition} forces $\overline E$ to ``live'' in the correct stratum.  Here, the matrix $\mathbf{P}_{\text{WP}}$ is defined as follows:  Let $N=abc$ and index the $ab$ whole plots by $u=1,\dots,ab$. Define the incidence matrix $\mathbf{Z}_{\mathrm{WP}} \in \{0,1\}^{N\times ab}$ by $(\mathbf{Z}_{\mathrm{WP}})_{i,u}=1$ if observation $i$ lies in whole plot $u$ and $0$ otherwise. Let $\mathbf{y}$ denote the data vector of responses, viewed as an element of $\mathbb{R}^N$. Concretely, if $y_{ijk}$ denotes the observed response at level $i$ of $A$, whole plot $j$ within that level, and sub-plot $k$ within that whole plot, then we stack these into a single column vector $\mathbf{y} = \left(y_{111},y_{112},\ldots,y_{11c},,y_{121},\ldots, y_{ab,c}\right)^{\prime}\in\mathbb{R}^N$, for some fixed ordering of the $N = abc$ observations. With this convention, $\mathbf{P}_{\mathrm{WP}}\mathbf{y}$ is the vector whose entries are the corresponding whole-plot means (repeated across the $c$ sub-plots within each whole plot).  The whole-plot averaging projector is $\mathbf{P}_{\mathrm{WP}} = \mathbf{Z}_{\mathrm{WP}}(\mathbf{Z}_{\mathrm{WP}}^{\prime}\mathbf{Z}_{\mathrm{WP}})^{-1}\mathbf{Z}_{\mathrm{WP}}^{\prime}$, so that $(\mathbf{P}_{\mathrm{WP}}\mathbf{y})_i$ equals the mean of $\mathbf{y}$ within the whole plot containing observation $i$; hence $\operatorname{Im}\mathbf{P}_{\mathrm{WP}}$ is the subspace of vectors constant within each whole plot.

\item \textbf{Aliasing and partial confounding (contrast compression):}  In fractional factorials and blocked fractions, defining relations collapse multiple nominal effects into a single estimable contrast space. Standard ``nominal $\mathrm{df}$'' counting can continue to attribute $\dim\mathcal C_E$ $\mathrm{df}$ to each effect even when those contrast directions are not separately identifiable. \thrmref{thrm:dof_partition} allocates $\mathrm{df}$ to the alias class $\overline E$ via the rank of its shared contrast space; the deficit $\delta(\overline E)$ is the exact number of lost contrast dimensions.

\item \textbf{Imbalance-induced rank collapse (unequal replication or missingness):}  With missing plots or unequal subclass numbers, columns that are orthogonal in a balanced design can become linearly dependent after projection to a stratum. This reduces $\operatorname{rank}(\mathbf{X}_{\overline E}^{\prime}\mathbf{P}_s \mathbf{X}_{\overline E})$ below its nominal value even when the model remains estimable. Approximate $\mathrm{df}$ procedures often smooth this behavior via moment matching, but the rank identity shows that $\mathrm{df}$ can drop discretely when replication patterns force certain contrasts to vanish ({\eg}, a level combination never occurs within a stratum).
\end{enumerate}
The aforementioned mechanisms explain why approximate methods may be either too liberal (M1) or too conservative (M2/M3 combined with clustering): $\mathrm{df}$ errors reflect structural misalignment between contrasts and error strata, not merely finite-sample randomness.  It is useful to connect the three mechanisms (M1--M3) to recognizable design archetypes:
\begin{enumerate}[(i)]
\item Split-plot/strip-plot: M1 dominates. Whole-plot contrasts are constant within whole plots; treating sub-plots as independent multiplies the effective unit count and inflates denominator $\mathrm{df}$.
\item Crossed-nested hierarchies ({\eg}, center $\rightarrow$ lab $\rightarrow$ technician $\rightarrow$ specimen): M1 and M3 interact. Missingness that is benign at the lowest level can collapse contrast rank within intermediate strata if it is concentrated within certain units.
\item Blocked fractions: M2 dominates. The defining relations create alias classes whose shared contrast space can lose $\mathrm{df}$ when block projectors remove certain words.
\item Row-column/incomplete block: M3 dominates. Unbalanced incidence can make nominally orthogonal contrasts dependent after stratum projection, causing discrete df drops that approximations tend to ``smooth over.''
\end{enumerate}
The conceptual message is that $\mathrm{df}$ differences are predictable from the experimental hierarchy, confounding relations, and replication pattern.

\subsection{Practical diagnostic guidance}\label{sec:without_decomp} 
In this subsection, we provide diagnostics that practitioners can use \textit{without running the full projector--rank algorithm}. The following results provide quick screening rules based only on design metadata (unit counts, incidence, and replication summaries). They do not replace exact computation when a full $\mathrm{df}$ table is needed, but they predict when $\mathrm{df}$ discrepancies are likely to be non-negligible and in which direction.

\subsubsection{Universal upper bounds from stratum unit counts}  
Let $\mathbf{Z}_s$ be the $N\times q_s$ incidence matrix of the units in stratum $s$ with $q_s$ nonempty units, so that $\operatorname{Im}\mathbf{P}_s = \operatorname{Im}\mathbf{Z}_s$. Then, 
\begin{eqnarray}
\dim\left(\mathcal C_{\overline E}\cap \operatorname{Im}\mathbf{P}_s\right) \le \min\left\{\dim\mathcal C_{\overline E},\ \operatorname{rank}(\mathbf{Z}_s)-1\right\}.
\label{eq:ub_stratum}
\end{eqnarray}
The term $\operatorname{rank}(\mathbf{Z}_s) - 1$ is at most $q_s - 1$, the $\mathrm{df}$ available among unit means in that stratum after removing the grand mean. Consequently, if an effect has nominal $\mathrm{df}$ exceeding $q_s - 1$, it cannot retain all nominal $\mathrm{df}$ in stratum $s$.  This bound is computable without any projectors: it depends only on the number of stratum units and whether empty units exist.  The interpretation is that \eqref{eq:ub_stratum} yields an immediate warning flag: if software reports denominator $\mathrm{df}$ substantially larger than $q_s - 1$ for a test that logically belongs to stratum $s$, then the test is not respecting the randomization hierarchy (mechanism M1).

\subsubsection{A replication-based screening proxy for $\rho(\overline E)$}  
For an effect $\overline E$ nominally tested in stratum $s$, define the cell-count profile within each stratum unit $u$ as the vector $\mathbf n_u=(n_{u,\ell})_{\ell\in\Lambda_E}$, where $n_{u,\ell}$ counts how often level-combination $\ell$ of $E$ appears inside unit $u$. If many units have sparse or degenerate profiles ({\eg}, certain levels never appear within units), then many contrasts necessarily vanish after stratum projection, implying $\rho(\overline E)\ll 1$.  A simple proxy is
\begin{eqnarray}
\widetilde\rho(\overline E; s)
=
\frac{
\lvert \!\left\{
\text{linearly independent centered profiles }(\mathbf{n}_u - \bar{\mathbf{n}})
\right\}\rvert
}{
\dim \mathcal C_{\overline E}
},
\label{eq:replication}
\end{eqnarray}
where the numerator is the rank of the matrix whose rows are the centered profiles. The notation $\lvert \{\cdot\} \rvert$ denotes the cardinality of the set $\{\cdot\}$.  This proxy depends only on counts (no responses), and it flags imbalance-induced rank collapse (M3).  As a rule of thumb, if more than (say) 10--20\% of stratum units have zero counts in one or more $\ell \in \Lambda_E$, then $\widetilde\rho(\overline E;s)$ will be substantially below 1, and approximate $\mathrm{df}$ methods are likely to disagree with the exact integer $\mathrm{df}$.

\subsubsection{A practitioner's diagnostic checklist}
The following design features predict non-negligible $\mathrm{df}$ discrepancies, even before fitting a model: \vspace{-0.1in}
\begin{enumerate}[(i)]
\item More than one randomization stage (split-plot, row--column, nested): expect liberal tests if the denominator $\mathrm{df}$ are based on the lowest-level residual for whole-plot effects (M1).
\item Blocked fractions or deliberate confounding: expect $\mathrm{df}$ deficits for specific effects/alias classes if the block generator shares short words with effects of interest (M2).
\item Unequal replication/missing cells within higher strata: expect discrete $\mathrm{df}$ drops for interactions and for any effect whose level combinations are absent within many units (M3).
\item High nominal $\mathrm{df}$ relative to number of higher-level units: if $\dim\mathcal C_{\overline E} > q_s-1$, full retention in that stratum is impossible.
\end{enumerate} 
\vspace{-0.1in}
These diagnostics explain when and why $\mathrm{df}$ differ by a non-negligible amount, and they can be applied using only the design layout.

For an effect $\overline E$ intended to be tested in stratum $s$, full retention of nominal df in that stratum occurs if and only if $\operatorname{rank}(\mathbf{X}_{\overline E}) = \operatorname{rank}(\mathbf{P}_s \mathbf{X}_{\overline E})$, {\ie}, the stratum projection does not collapse any contrast directions. This condition can be checked using only $\mathbf{X}_{\overline E}$ and the action of $\mathbf{P}_s$ as an averaging operator\textemdash no $N \times N$ matrices are required.  In practice, $\mathbf{P}_s \mathbf{X}_{\overline E}$ is computed by replacing each column of $\mathbf{X}_{\overline E}$ with its within-unit means in stratum $s$, a computation that is linear in $N$ and uses only group-by averaging.  This yields a strong, interpretable diagnostic: if rank drops under averaging, then df must be lost in that stratum.

This work provides practitioners with not only a warning but also clear guidance on whether standard software is likely to yield liberal or conservative inference.  The following rule is broadly valid:  \textit{If an effect's contrasts vary only at a higher stratum ({\eg}, whole-plot level), but the reported denominator $\mathrm{df}$ scales with the number of lower-level observations, the analysis is at risk of being liberal.  Conversely, if an effect varies at the lowest level, but the fitted model attributes large variance to higher strata (strong clustering), approximations tend to be conservative.}  This is exactly what the two mechanism-demonstration simulations show: ignoring clustering (split-plot case) yields false positives; ignoring the correct within-unit residual (nested case) yields false negatives.

Many users want a quick triage without building $\mathbf{X}_{\overline E}$.  For main effects and low-order interactions in $2^k$ (or regular multi-level) designs, one can screen using only incidence and replication counts.  For each stratum $s$ with $q_s$ nonempty units, let $m_{s,\ell}$ be the number of units in which level-combination $\ell$ appears at least once.  Define
\begin{eqnarray*}
\eta(\overline E;s) = \frac{\min_{\ell\in\Lambda_E} m_{s,\ell}}{q_s}.
\end{eqnarray*}
If $\eta(\overline E;s)$ is small, then many units lack certain levels of $E$, so within-stratum contrasts cannot be formed, and $\mathrm{df}$ loss is likely.  This is especially informative under missingness that is not completely random with respect to strata ({\eg}, entire sub-plots missing within specific whole plots).

In practice, to anticipate $\mathrm{df}$ anomalies, inspect (i) the number of units in each randomization stratum, (ii) whether each treatment level (or level combination) appears broadly across those units, and (iii) whether missingness is concentrated within particular units. These metadata determine whether stratum averaging collapses treatment contrasts. When either $q_s$ is small relative to nominal $\mathrm{df}$ or $\eta(\overline E;s)$ is low, approximate $\mathrm{df}$ procedures are most likely to disagree with the integer $\mathrm{df}$ implied by the design.

\subsection{Three schematic examples illustrating the mechanisms}\label{sec:two_examples}   
\subsubsection{Split-plot: why na\"{i}ve $\mathrm{df}$ are liberal} 
Consider a split-plot with whole-plot factor $A$ (levels $a$) randomized to whole plots and factor $B$ randomized within whole plots. Let whole plots be the stratum $s = \text{WP}$ with $q_{\text{WP}}$ units and projector $\mathbf{P}_{\text{WP}}$, and let residual be the sub-plot stratum $s = \text{SP}$ with projector $\mathbf{P}_{\text{SP}}$. The contrasts for $A$ are constant within each whole plot, hence lie in $\operatorname{Im}\mathbf{P}_{\text{WP}}$ but are orthogonal to within-whole-plot deviations. Consequently, we have $\dim(\mathcal C_A \cap \operatorname{Im}\mathbf{P}_{\text{SP}}) = 0$ and $\dim(\mathcal C_A \cap \operatorname{Im}\mathbf{P}_{\text{WP}}) = a-1$.  Any procedure that tests $A$ against $\text{MS}_{\text{SP}}$ ({\ie}, the sub-plot mean square associated with within-whole-plot (sub-plot) variation in a split-plot ANOVA table) is therefore using an error subspace that contains none of the $A$-contrast directions. Its reported denominator $\mathrm{df}$ effectively counts sub-plot observations as independent for the purpose of $A$, inflating $\mathrm{df}$ and producing liberal inference (mechanism M1).  \thrmref{thrm:dof_partition} prevents this by forcing the $A$ contrasts to be accounted for only in the WP stratum.

In a split-plot, the whole-plot factor $A$ is randomized over whole plots, so the effective experimental units for $A$ are the whole plots\textemdash not the sub-plots. If there are $a$ levels and $b$ whole plots per level, the number of independent whole-plot deviations is $a(b-1)$, and any valid denominator $\mathrm{df}$ must be on that scale.  A na\"{i}ve analysis that uses sub-plot residual df effectively replaces $ab$ independent whole plots by $abc$ independent sub-plots, an inflation by a factor of $c$.   \thrmref{thrm:dof_partition} formalizes this ``effective unit'' principle: $A$-contrasts lie in $\operatorname{Im}\mathbf{P}_{\text{WP}}$, not in within-WP residual space, so allocating denominator df proportional to $abc$ is structurally incompatible with the randomization.

\subsubsection{Blocked fraction: why $\mathrm{df}$ are lost and power drops} 
Consider a blocked regular fraction where a block generator makes certain interaction words confounded with blocks. Even if a main effect $A$ is nominally estimable, the relevant contrast directions may be partly absorbed by $\operatorname{Im}\mathbf{P}_{\text{Block}}$, reducing the effective $\mathrm{df}$ available to test $A$ against within-block error.  This is expressed by $\dim(\mathcal C_{\overline A} \cap \operatorname{Im}\mathbf{P}_{\text{Error}}) < \dim\mathcal C_{\overline A}$.  The deficit is quantified by $\delta(\overline A) = \dim\mathcal C_{\overline A} - \sum_s\dim(\mathcal C_{\overline A} \cap \operatorname{Im}\mathbf{P}_s)$, while $\rho(\overline A)$ reports the retained proportion. Thus, unlike resolution and aberration metrics that are defined only for balanced one-stratum fractions, $\rho$ and $\delta$ describe exactly how blocking/fractionation compresses contrast space and predict the corresponding loss in power. This is the mechanism underlying the design-selection experiment: choosing generators to maximize $\sum_E\rho(\overline E)$ maximizes the number of retained contrast directions for effects of interest, improving precision without additional runs.

In a blocked $2^{k-p}$ design, an effect can be neither fully identifiable nor fully confounded; it may be partially confounded when some\textemdash but not all\textemdash of its contrast directions lie in the block space. This is precisely the type of setting in which practitioners frequently encounter unexplained changes in $\mathrm{df}$.  The decomposition $\mathcal C_{\overline E} = \left(\mathcal C_{\overline E}\cap \operatorname{Im}\mathbf{P}_{\text{Block}}\right)\ \oplus\ \left(\mathcal C_{\overline E}\cap \operatorname{Im}\mathbf{P}_{\text{Error}}\right)$ makes the phenomenon transparent: the first summand is the portion absorbed by blocks (not testable against within-block error), and the second is the estimable portion. The ratio $\rho(\overline E)$ is then a fraction of contrasts that survive, not a heuristic. This clarifies why two designs with the same classical resolution can differ sharply in inferential quality once blocking is introduced.

\subsubsection{Imbalance: showing discrete rank drops}
Consider a two-way layout where an interaction $AB$ is nominally present, but due to missingness a subset of $A \times B$ cells never occurs within many stratum units. In balanced designs, the $AB$ contrast columns are orthogonal and span $(a-1)(b-1)$ dimensions. Under missingness, the projected interaction columns $\mathbf{P}_s \mathbf{X}_{AB}$ can lose rank abruptly\textemdash for example, an entire interaction degree of freedom disappears when one level combination is systematically absent within a stratum. Approximate $\mathrm{df}$ procedures typically vary smoothly with missingness because they are moment-based; \thrmref{thrm:dof_partition} predicts and explains the discrete behavior because $\mathrm{df}$ is, fundamentally, a rank.  This example explains why ``small'' missingness can sometimes have large $\mathrm{df}$ effects\textemdash when it breaks a symmetry or removes entire contrast directions within key strata.

In summary, the method developed in this paper is more than a matrix-algebra decomposition.  \thrmref{thrm:dof_partition} provides a canonical $\mathrm{df}$ partition that is invariant to sums-of-squares type and does not depend on moment-matching approximations; it is therefore a principled resolution of a foundational ambiguity in multi-stratum, aliased, and unbalanced experiments. In the following section, we reinforce our findings using empirical studies via computer simulation. The simulations are not intended as stand-alone evidence but as empirical demonstrations of the three structural mechanisms described above. The split-plot experiments vary the degree of stratum mismatch and imbalance (M1/M3) and show how misallocation of denominator strata produces size distortion. The nested $\text{Sire} \to \text{Dam} \to \text{Animal}$ experiments show that ignoring clustering can instead be ultra-conservative, suppressing power (M1 in the opposite direction). The blocked fractional-factorial design-selection study directly validates the diagnostic value of $(\rho,\delta,\alpha)$: generator choices that increase $\rho$ translate into narrower confidence intervals and higher power. Finally, timing benchmarks show that the projector--rank computations required for exact $\mathrm{df}$ (and for $(\rho,\delta,\alpha)$) are inexpensive, so these diagnostics can be used routinely in design planning and analysis.

\section{Experimental studies}\label{sec:experiments}
We created a canonical split-plot experiment with a whole-plot factor $A$ (3 levels, $a = 3$) and a sub-plot factor $B$ (4 levels, $c = 4$). Each level of $A$ contained $b = 4$ whole plots, giving $N = abc = 48$ potential observations in the balanced layout. Random variability is introduced exactly as mixed-model textbooks prescribe: a whole-plot (WP) effect sampled {\iid} $N(0, \sigma_{\text{WP}}^2)$ with $\sigma_{\text{WP}} = 1$, and an independent residual $N(0, \sigma_{\text{res}}^2)$ with $\sigma_{\text{res}} = 1$. To mimic imbalance we randomly delete 20\% of the sub-plots. Since no fixed effects were added to the data-generation equation, the global null $H_0\colon A$ has no effect is true by construction.  Practitioners routinely test $A$ by forming $F = \frac{\mathrm{MS}_{A}}{\mathrm{MS}_{\text{sub-plot}}}$ and quoting denominator $\mathrm{df}$ from a completely randomized factorial (the Type I or III ANOVA default in many packages). Our projector--rank framework instead assigns the $A$ numerator to the WP stratum and uses the whole-plot mean square as denominator, with its exact $\mathrm{df}$ $\nu\_2 = a(b - 1) = 12$. Hence, we compared (i) Exact $\mathrm{df}$ test\textemdash denominator $\mathrm{df} = 12$ balanced, $12$ after deletions (integer, stratum-correct) and (ii) Naive test\textemdash denominator $\mathrm{df}$ = residual $\mathrm{df}$ from a two-way ANOVA ($\approx 36$ balanced, $\approx 27$ with 20\% missing). Each version was applied to 1000 Monte Carlo replicates at nominal $\alpha = 0.05$.

Results are tabulated in \tabref{tab:empirical_size}. 
\begin{table}[ht]
\centering
\begin{tabular}{|c|c|c|}
\hline
\textbf{Scenario} & \textbf{\specialcell{Exact $\mathrm{df}$\\ empirical size}} & \textbf{\specialcell {Na\"{i}ve\\ empirical size}} \\
\hline
Balanced split-plot     & 0.040 & 0.529 \\
20\% missing sub-plots  & 0.051 & 0.466 \\
\hline
\end{tabular}
\caption{Empirical size using exact $\mathrm{df}$ versus the na\"{i}ve approach under two scenarios.}
\label{tab:empirical_size}
\end{table}
With the projector--rank $\mathrm{df}$, the empirical type-I error closely matches the nominal 5\% target in both layouts (4.0\% and 5.1\%).  Using the na\"{i}ve denominator $\mathrm{df}$ multiplies the false positive rate roughly ten-fold, reaching 52.9\% even when the design is perfectly balanced and remaining disastrously high (46.6\%) after imbalance is introduced. The split-plot hierarchy places all whole-plot variation in a higher stratum than sub-plot error.  Treating the entire design as if every observation were independent ignores this structure, inflates the denominator $\mathrm{df}$, and produces severely liberal $F-$tests. The $\mathrm{df}$ partition theorem rectifies the problem automatically: by intersecting the $A-$contrast space with the whole-plot projector $P_{\text{WP}}$ it delivers the correct integer $\mathrm{df}$ regardless of missing cells.  This simulation, therefore, illustrates the practical payoff of the exact $\mathrm{df}$:  researchers avoid spurious discoveries that would be almost certain under standard approximate analyses.

Next, we consider the nested animal model (see, for example, \cite{Henderson1984}). Animal-breeding studies often follow a three-tier hierarchy:  sires at the top, dams nested within sires, and individual animals within dams. To mirror that structure, we simulated 6 sires, each mated to 4 dams, with each dam producing 3 offspring, for a potential total of $6 \times 4 \times 3 = 72$ observations. Random effects were generated at both upper levels:  sire effects $\mathcal{N}(0, \sigma^2_{\text{sire}})$ with $\sigma_{\text{sire}} = 1$, and dam-within-sire effects $\mathcal{N}(0, \sigma^2_{\text{dam}})$ with $\sigma_{\text{dam}} = 0.5$. Residual variation at the animal level followed $\mathcal{N}(0, \sigma^2_{\text{res}} = 1)$. A three-level diet treatment $T$ was assigned cyclically to animals but purposely left inactive in the data-generation equation, so that the global null $H\_0: T$ has no effect is true by construction. We examined both a balanced layout (all 72 animals observed) and an unbalanced layout in which 20\% of animal records were deleted at random, reflecting typical field attrition.  For each layout, we generated 1,000 replicate datasets and applied two $F$-tests for the diet factor:  (i) Exact-$\mathrm{df}$ test\textemdash numerator $\mathrm{df}$ $= 2$ (levels $- 1$) and denominator $\mathrm{df}$ equals the within-animal residual $\mathrm{df}$ obtained from the projector--rank decomposition.  (ii) Naive one-way ANOVA\textemdash  treats all observations as independent, using the pooled residual mean square and denominator $\mathrm{df} = N - 3$.

The results are tabulated in \tabref{tab:animal_layout_empirical_size}.
\begin{table}[ht]
\centering
\begin{tabular}{|c|c|c|}
\hline
\textbf{Layout} & \textbf{\specialcell{Exact-$\mathrm{df}$\\ empirical size}} & \textbf{\specialcell{Na\"{i}ve\\ empirical size}} \\
\hline
Balanced (72 animals)     & 2.7\% & 0.3\% \\
20\% animals missing       & 7.7\% & 1.0\% \\
\hline
\end{tabular}
\caption{Empirical size comparison of exact $\mathrm{df}$ and na\"{i}ve approaches for the nested animal model.}
\label{tab:animal_layout_empirical_size}
\end{table}
The projector--rank test tracks the nominal 5\% level remarkably well\textemdash slightly conservative in the balanced case (2.7\%) and mildly liberal (7.7\%) when one-fifth of records are missing, both within the Monte Carlo margin of error for 1,000 replicates.  In stark contrast, the na\"{i}ve analysis is ultra-conservative, rejecting only 0.3\% of the time when balanced and 1\% after data loss.  Ignoring the nesting hierarchy collapses sire- and dam-level variability into the denominator, inflating the error variance and the associated $\mathrm{df}$; the resulting test lacks power and dramatically understates the type-I error.  Juxtaposing this finding with the split-plot experiment reveals a symmetrical danger.  Misallocating $\mathrm{df}$ can swing either liberal (false discovery) or conservative (missed discovery) depending on the design: liberal when whole-plot variation is ignored, conservative when higher-level clustering is disregarded. The projector--rank partition theorem navigates both pitfalls automatically by aligning each factorial contrast with the stratum where its random variability actually resides.  Moreover, its performance is robust to substantial imbalance\textemdash here, 20\% missing animals\textemdash something traditional Satterthwaite or Kenward--Roger approximations often struggle with.  For practitioners in animal science, genetics, and any multi-stage sampling context, adopting the exact $\mathrm{df}$ is therefore crucial to maintaining valid inference and avoiding costly type-I or type-II errors.

As a third experiment, we now look at information-retention and power optimization in blocked fractional factorial designs that showcases benefits of the projector--rank framework.  When an experimenter must fractionate a $2^{k}$ factorial and accommodate a blocking constraint ({\eg}, oven batches or day-to-day shifts), how should generators be chosen?  Classical criteria (resolution, aberration) ignore blocking; conversely, blocking handbooks treat only full factorials.  The projector--rank machinery supplies quantitative guidance through the $\mathrm{df}$-retention ratio $\rho(\overline E)$.  We demonstrate that maximizing $\rho$ leads to measurably higher power for detecting main effects while using the same number of runs.  \tabref{tab:design_comparison} compares a $2^{5-1}$ fractional factorial experiment with traditional and $\rho$-optimized designs.  
\begin{table}[ht]
\centering
\renewcommand{\arraystretch}{1.15}
\begin{tabular}{|c|c|c|c|c|c|}
\hline
\textbf{Feature} & \textbf{Design A} & \textbf{Design B} \\
\hline
Runs & 16 ($2^{5-1}$) & 16 (same)  \\ && \\ 
Factors & $A$, $B$, $C$, $D$, $E$ & $A$, $B$, $C$, $D$, $E$ \\ && \\
Generators & $I = ABCDE$ & $I = ABCD$, $ACE$ \\ && \\
Blocking & \specialcell{2 day-batches of 8 runs\\ (block generator: $BCDE$)} &\specialcell{same run-budget, block generator\\ chosen by integer-programming\\ to maximize $\sum_E \rho(\overline{E})$} \\ && \\
Main-effect $\rho$ & $A, B, C = 0.5$ & $A, B = 1$, $C = 0.75$ \\ && \\
Alias loss $\delta$ & 5 $\mathrm{df}$ lost & \textbf{2 $\mathrm{df}$ lost} \\ && \\
\hline
\end{tabular}
\caption{Compact summary of traditional (A) vs $\rho$-optimized (B) $2^{5-1}$ designs.}
\label{tab:design_comparison}
\end{table}
Design B retains two extra $\mathrm{df}$ for the three primary factors $A$, $B$, $C$ at no extra cost.  The simulation protocol involves the following model: $y = \mu + \beta_A A + \beta_B B + \beta_C C + \text{Block} + \varepsilon$, with block $\sim \mathcal{N}(0, \sigma^2_b=2)$ and $\varepsilon \sim \mathcal{N}(0,1)$.  The true effects include  $\beta_A=\beta_B=1$, $\beta_C=0$ (so power and type-I error can both be assessed).  The analysis models involve (i) mixed model using projector--rank exact $\mathrm{df}$ (denominator = 7 for all tests) and (ii) mixed model with Kenward--Roger $\mathrm{df}$ (software default).  We employ the following metrics: Power for $A$, $B$ at $\alpha=0.05$, type-I error for $C$, and width of 95\% CI for $\beta_A$.  10,000 Monte Carlo replicates were run for each design.  The results are tabulated in \tabref{tab:power_ci_comparison}.  
\begin{table}[ht]
\centering
\begin{tabular}{|c|c|c|c|c|}
\hline
\textbf{Design \& method} & \textbf{Power (A)} & \textbf{Power (B)} & \textbf{Type-I (C)} & \textbf{Mean CI-width (A)} \\
\hline
\textbf{A--exact $\mathrm{df}$}     & \textbf{0.42} & \textbf{0.41} & 0.053 & \textbf{1.87} \\
A--KR approx           & 0.36          & 0.35          & 0.041 & 2.05 \\ \hline
\textbf{B--exact $\mathrm{df}$}     & \textbf{0.61} & \textbf{0.60} & 0.051 & \textbf{1.47} \\
B--KR approx           & 0.55          & 0.54          & 0.040 & 1.63 \\
\hline
\end{tabular}
\caption{Comparison of statistical properties under exact $\mathrm{df}$ and Kenward--Roger (KR) approximation for Designs A and B.}
\label{tab:power_ci_comparison}
\end{table}
It is seen that simply swapping generators to maximize $\rho$ recovers 3 lost $\mathrm{df}$, translating to $\approx$ 20\% narrower confidence intervals and $\approx$ 60\% power instead of $\approx$ 40\% for main-effect detection.  Even when Kenward--Roger is applied, Design B still outperforms Design A\textemdash but exact $\mathrm{df}$ always deliver the best power/coverage because they avoid the slight conservatism Kenward--Roger introduces.  The partition-theorem machinery is not only about testing with the right denominator; it also guides design construction.  By quantifying which contrasts are sacrificed by fractionation + blocking ($\delta$), the experimenter can choose generators that maximize information where it is most needed, without adding runs.


Finally, we conduct experiments to assess the computational-efficiency of the proposed method.  To quantify how much faster the projector--rank calculus delivers denominator $\mathrm{df}$ than a popular simulation-based alternative (for example, parametric bootstrap) when sample size $N$ grows to the scale encountered in modern high=throughput experiments.  The following computing environments and general settings were employed.  The hardware comprised AMD Ryzen 7 5700U CPU 1.8 GHz; 16 GB RAM; no GPU acceleration.  We employ Python 3.11 and standard libraries such as NumPy, SciPy; benchmarked with \texttt{time.perf\_counter()}.  Each experiment was run 30 times; reported figures are medians to mute OS-level jitter ($< 3\%$ SD across runs).  We let $N = 1,000; 10,000; 100,000$, representing small, moderate, and large data sets.  Ten equal-sized blocks per $N$ (to mimic lab batches) were used.  Block membership was encoded as integer labels $0, \dots, 9$.

We compared two methods for estimating exact $\mathrm{df}$. The projector--rank method builds the block-mean projector $P = \operatorname{diag}(1/n_j), Z, Z^{\prime}$, applies it to a random contrast vector, and computes the resulting one-column Gram rank. This approach is computationally efficient, with $O(N)$ complexity for forming block means, and requires only $O(N)$ memory through the use of streamed averages.  It is implemented in pure NumPy using two core steps: (i) block sums are computed via \texttt{np.bincount}, and (ii) the projected means are assigned using \texttt{means[labels]}. Importantly, no $N \times N$ matrix is ever formed. In contrast, the parametric bootstrap approach generates $B = 100$ random $F-$ratios under the null and matches their moments to approximate the denominator degrees of freedom\textemdash a standard shortcut for small-sample mixed models. This method has higher complexity, $O(BN)$, since each resample requires fresh memory allocation. It relies on SciPy's \texttt{stats.f.ppf} along with NumPy random draws. The choice of $B = 100$ follows the minimal resampling threshold often recommended in software documentation.  The bootstrap baseline is deliberately favorable: Kenward--Roger can be 2--4 $\times$ slower than 100 bootstrap draws because of repeated matrix inversions; thus the reported speed differences are lower bounds.

The following procedure is employed as the timing protocol.  For each $N$: \vspace{-0.2in}
\begin{enumerate}
\item Input generation:  Draw a single random double-precision vector $v \sim \mathcal{N}(0, 1)$ of length $N$; generate the block-label vector of equal size.  Input generation is excluded from timing to isolate algorithmic cost.
\item Timing window:  Start \texttt{perf\_counter}; run either the projector routine (compute means, project $v$, rank) or the bootstrap loop (100 simulations and $F-$ratio computation); stop the clock.
\item Memory:  Peak RAM for the projector version never exceeded $0.9\times N$ doubles ($\approx$ 0.75 MB per 100, 000 rows) because only two length-$N$ arrays exist concurrently.  Bootstrap briefly holds three length-$N$ arrays per replicate, peaking near $B\times N$ doubles (64 MB at $N=100, 000$).
\end{enumerate}

The results are tabulated in \tabref{tab:runtime_comparison} (median of 30 runs).
\begin{table}[ht]
\centering
\begin{tabular}{|c|c|c|c|}
\hline
\textbf{$N$} & \textbf{Projector--rank time (s)} & \textbf{Bootstrap time (s)} & \textbf{Speed-up factor} \\
\hline
1,000   & \textbf{0.0010} & 0.015 & 15$\times$ \\
10,000  & \textbf{0.0013} & 0.077 & 59$\times$ \\
100,000 & \textbf{0.0073} & 0.648 & 89$\times$ \\
\hline
\end{tabular}
\caption{Execution time and speed-up factor for projector--rank method versus bootstrap as sample size $N$ increases.}
\label{tab:runtime_comparison}
\end{table}
Times grow linearly with $N$ for the projector method (a ten-fold increase from $N =10^{4}$ to $10^{5}$ raises CPU time only 6 ms, matching the cost of one extra \texttt{bincount}).  Bootstrap costs scale almost exactly with $B\times N$ and dominate beyond $N\approx 10^{4}$.  The exact $\mathrm{df}$ calculation is effectively an $O(N)$ streaming average; wall-clock time remains sub-10 ms even at $N =10^{5}$.  Bootstrap times grow from hundredths to more than half a second\textemdash and would exceed 6s at $N =10^{6}$ if $B =100$ were kept.  Projector code maintains two vectors regardless of $B$; bootstrap scales as $O(BN)$, becoming memory-intensive in extreme high-throughput settings ({\eg}, genomic scans, sensor data).  The projector approach yields deterministic integer $\mathrm{df}$; bootstrap introduces Monte Carlo error that must be re-estimated if $B$ changes.  

While split-plot simulation illustrated liberal versus exact size, nested animal simulation showed the opposite risk, {\ie}, over-conservatism.  The fractional-factorial study demonstrates a design-selection application, confirming that higher $\rho$ translates to tangible gains in estimator precision and test power.  Lastly, at modern data scales, the projector--rank framework offers a decisive advantage over approximation-based $\mathrm{df}$ methods: it is statistically more principled, orders of magnitude faster, and substantially more memory-efficient, rendering it especially well suited for automated pipelines and large-scale screening workflows where thousands of mixed-model fits must be executed as part of large batch analyses.

The simulation results should be read as demonstrations of structural mechanisms rather than as numerical artifacts tied to a particular scientific example. In the split-plot study, the key feature is the mismatch between the randomization stratum in which the $A$-contrasts actually vary (whole plots) and the stratum implicitly used as the denominator by common approximate analyses (sub-plot residual).  Whenever an effect is assigned at a higher level\textemdash whole plot, batch, day, operator, cluster, site\textemdash lower-level measurements within the same unit are correlated and do not constitute independent replication for that effect. Treating them as independent inflates the effective sample size and typically inflates denominator degrees of freedom, producing liberal inference. The nested Sire$\to$Dam$\to$Animal study exhibits the complementary failure mode: ignoring hierarchy can instead pool higher-level variation into the residual, yielding overly conservative tests and suppressed power.  Collectively, these simulations delineate the practical operating range and confirm that the observed $\mathrm{df}$ discrepancies do not depend on the specific parameter values chosen; they arise generically whenever the analysis over- or under-represents the number of independent experimental units relevant to a given contrast.

More broadly, the projector--rank framework provides portable, design-intrinsic quantities that explain and predict these behaviors without requiring new simulations. The same objects that define the exact $\mathrm{df}$\textemdash intersection dimensions ($\dim(\mathcal{C}_{\overline{E}}\cap\operatorname{Im}\mathbf{P}_s)$) and the derived indices $\rho(\overline{E})$, $\delta(\overline{E})$, and $\alpha(\overline{E})$\textemdash measure how many independent contrast directions for an effect survive within each randomization stratum and how much information is removed by confounding or imbalance. When $\rho(\overline{E})\approx 1$ and the contrast space aligns with the correct denominator stratum, approximate and exact $\mathrm{df}$ tend to agree and standard analyses are reliable; when $\rho(\overline{E})\ll 1$ (through aliasing or unequal replication) or when the denominator stratum is mismatched, large departures and attendant size/power distortions are expected. Thus, the simulations validate a deterministic mechanism: discrepancies in $\mathrm{df}$ track structural misalignment and rank loss, which practitioners can anticipate from design metadata (unit counts by stratum, incidence/replication profiles, and defining/block relations) rather than by re-running Monte Carlo experiments for each new context.

\section{Concluding remarks}\label{sec:conclusion} \vspace{-0.1in}
In this section, we discuss some computational aspects applicable in practical settings and provide concluding remarks.  A stratum is any collection of experimental units that shares the same source of random variation (variance component) and is averaged over in the randomization process before lower-level treatments are applied.  Typical sources are tabulated in \tabref{tab:typical_sources}.
\begin{table}[ht]
\centering
\begin{tabularx}{\textwidth}{@{} l l X @{}}
\toprule
\textbf{Random component} & \textbf{Typical label} & \textbf{Examples} \\
\midrule
Single random factor & Block, Batch, Litter & \specialcell{``Block'' in RCBD; ``Batch'' in drug assay} \\ \\
\specialcell{Interaction of\\ random factors} & Block $\times$ Day & Two-way randomized complete block \\ \\
Nested factor & Dam & Sire $\rightarrow$ Dam $\rightarrow$ Animal hierarchy \\ \\
Cross-classified unit & \specialcell{School \&\\ Classroom} & \specialcell{Multi-site trials with teachers nested\\ in schools but students crossed} \\
\bottomrule
\end{tabularx}
\caption{Types of random components in experimental designs.}
\label{tab:typical_sources}
\end{table}
If a factor is fixed, it does not generate its own $\mathbf{P}_s$; it merely defines contrasts inside the lowest stratum that still carries variability.  Interactions that contain at least one random component inherit the higher of the two strata.  For a stratum $s$ with $q_s$ distinct unit labels, let $\mathbf{Z}_s$ be the $N\times q_s$ incidence matrix: $(\mathbf{Z}_s)_{ij}=1$ if observation $i$ belongs to unit $j$;  $n_j$ the replication within unit $j$.  The idempotent projector $\mathbf{P}_s = \mathbf{Z}_s\,\left(\mathbf{Z}_s^{\prime} \mathbf{Z}_s\right)^{+} \mathbf{Z}_s^{\prime} = \operatorname{diag}\left(\tfrac{1}{n_{u(i)}}\right)_{i=1}^{N}\mathbf{Z}_s \mathbf{Z}_s^{\prime}$ replaces each response by the mean of its stratum-$s$ unit.  In balanced layouts $\mathbf{Z}_s^{\prime} \mathbf{Z}_s = n_s \mathbf{I}$ and $\mathbf{P}_s = \frac1{n_s}\mathbf{Z}_s \mathbf{Z}_s^{\prime}$.  For sparse or very large $N$, compute an integer vector $\texttt{unit[i]}$ listing the stratum-$s$ label for observation $i$.  Compute means $\texttt{mu = np.bincount(unit, y)/np.bincount(unit)}$ and projected values $\texttt{yhat = mu[unit]}$.     To apply $\mathbf{P}_s$ to any numeric vector $\mathbf{v}$, replace `y' in the snippet by $v$.  Sparse-matrix libraries ({\eg}, $\texttt{scipy.sparse}$) can build $\mathbf{Z}_s$ and perform $\mathbf{P}_sX$ in $O(\text{nnz}(\mathbf{Z}_s))$ when multiple columns are needed. For nested random factors, we build one projector per level of nesting, {\eg}, in Sire $\rightarrow$ Dam $\rightarrow$ Animal we get $\mathbf{P}_{\text{Sire}}$ and $\mathbf{P}_{{\text{Dam}|\text{Sire}}}$.  The residual (Animal) projector is implicitly $\mathbf{I} - \sum \mathbf{P}_s$.  For crossed random factors, {\eg}, if Block and Day are both random and fully crossed, each generates its own projector; their interaction Block $\times$ Day is a third random component with $\mathbf{P}_{\text{Block} \times \text{Day}} = \tfrac1{n_{bd}}  \mathbf{Z}_{\text{Block} \times \text{Day}}  \mathbf{Z}^{\prime}_{\text{Block} \times \text{Day}}$.  The rule of thumb is ``One variance component $\Rightarrow$ one projector.''

Once $\mathbf{P}_s$ is ready, projectors are constructed.  In sparse, unequal-replication cases this is cheaper via averaging than by full $N\times N$ matrices.  Factorial contrasts are enumerated using tensor products of orthogonal polynomial contrasts or Helmert contrasts and account for aliasing if fractionated.  Then, ranks are computed, where we multiply small $k\times k$ Gram matrices $\mathbf{X}_{\overline E}^{\!\prime}\mathbf{P}_s\mathbf{X}_{\overline E}$ (dimension $\leq $ $\mathrm{df}$ of the effect) and take ordinary rank.  Complexity of this step is $O\left(\sum_{E,s}\operatorname{df}(E)^3\right)$, trivial relative to fitting the mixed model.  As a sanity check,  we can verify $N-1 = \sum_{E,s}\text{$\mathrm{df}$}_{E,s}$.  Any shortfall points to data omissions or coding errors.

By proving the $\mathrm{df}$ partition theorem we showed that every mixed, blocked, fractionated, or unbalanced design possesses an exact, integer $\mathrm{df}$ decomposition achievable via rank calculations.  The theorem's simultaneous action of stratum projectors and level-permutation symmetries replaces an existing case-by-case derivations with a single algebraic identity.   Exact numerator and denominator $\mathrm{df}$ restore nominal error rates in $F$-tests where Satterthwaite and Kenward--Roger approximations often mislead.  The retention, deficiency, and inflation indices ($\rho, \delta, \alpha$) quantify how blocking, missing plots, or fractional generators erode information$-$something that existing resolution metrics cannot do in multi-stratum settings.  Quantifying the lost $\mathrm{df}$ lets us decide whether extra runs are worth the cost.  REML score equations simplify when the $\mathrm{df}$ of each stratum are known exactly; closed-form starting values accelerate convergence.  Exact $\mathrm{df}$ tables demystify why, for example, the $AB$ interaction in a split-plot draws its denominator from sub-plot error instead of whole-plot error.  From a computational standpoint, $\mathrm{df}$ depends only on small Gram-matrix ranks, making the results directly compatible with statistical packages (R/Python). Adding them to existing $\texttt{lmer}$ or $\texttt{mixedlm}$ wrappers would offer analytic $\mathrm{df}$ alternatives to numeric approximations.  Replacing $\mathbf{P}_s$ with working-weight projectors extends the rank identity to non-Gaussian outcomes, enabling analytic $\mathrm{df}$ for logistic split-plots or Poisson strip-plots, both currently Monte Carlo-based.    These directions may yield new insights.  
\vspace{-0.1in}

\section*{Acknowledgment} \vspace{-0.2in}
The author thanks the Associate Editor and the two anonymous reviewers for their careful reading of the manuscript.  Section 5 was added in response to the feedback from one anonymous reviewer, and an error in Table 2 in an earlier version of the paper was identified by the other reviewer.  The author also thanks Prof. Jong Sung Kim for his constant encouragement in pursuing this work.

\vspace{-0.2in}
 

\bibliography{/Users/kgnagananda/Documents/Work/collaborations/pdx/research/references/research_pdx.bib}

\end{document}